\documentclass[12pt,twoside]{amsart}
\usepackage[margin=3cm]{geometry}
\usepackage[colorlinks=false]{hyperref}
\usepackage[english]{babel}
\usepackage{graphicx,titling}
\usepackage{float}
\usepackage{amsmath,amsfonts,amssymb,amsthm}
\usepackage{lipsum}
\usepackage[T1]{fontenc}
\usepackage{fourier}
\usepackage{color}
\usepackage[latin1]{inputenc}
\usepackage{esint}
\usepackage{caption}
\usepackage{ccicons}

\makeatletter
\def\blfootnote{\xdef\@thefnmark{}\@footnotetext}
\makeatother

\newcommand\ccnote{
    \blfootnote{\ccLogo\, \ccAttribution\,\, Licensed under a Creative Commons Attribution License (CC-BY).}
}

\usepackage[export]{adjustbox}
\numberwithin{equation}{section}
\usepackage{setspace}
\renewcommand{\le}{\leqslant}

\renewcommand{\ge}{\geqslant}

\renewcommand{\mathbb}{\varmathbb}
\usepackage{fancyhdr}
\newtheorem{theorem}{Theorem}[section]
\newtheorem{lemma}[theorem]{Lemma}
\newtheorem{corollary}[theorem]{Corollary}
\newtheorem{proposition}[theorem]{Proposition}
\newtheorem{definition}[theorem]{Definition}
\newtheorem{remark}[theorem]{Remark}
\newtheorem{eg}{Example}[section]
\newtheorem{notation}{Notation}[section]

\def\R {\mathbb{R}}
\def\eps{\varepsilon}

\def\Sph{\mathbb{S}^{d-1}}

\def\Reg{\mathbf{Reg}}
\def\SOne{\mathbf{Sing}_1}
\def\STwo{\mathbf{Sing}_2}

\def\Trip{(u_1,u_2,u_3)}

\def\Subsol{\underline{\mathcal{A}}}
\def\Supsol{\overline{\mathcal{A}}}
\def\Sol{\mathcal{A}}

\def\Pabab{P(\alpha;a)}
\def\Qabab{Q(\beta;b)}
\def\Pab{P(\alpha)}
\def\Qab{Q(\beta)}
\def\PQabab{(P,Q)(\alpha,\beta;a,b)}

\def\PababTwo{P(\alpha,\beta;a,b)}
\def\QababTwo{Q(\alpha,\beta;a,b)}

\def\Sabab{\mathcal{S}(\alpha,\beta;a,b;\eps)}
\def\Sab{\mathcal{S}(\alpha,\beta;\eps)}

\def\PP{(\Phi,\Psi)}

\def\TopHalfTwo{\frac{1}{2}\min\{x\cdot\alpha-a,0\}^2+\frac{1}{4}\max\{x\cdot\beta-b,0\}^2}

\def\hu{\hat{u}}

\address{Ovidiu Savin, Department of Mathematics, Columbia University, 2990 Broadway, New York, NY}
\email{savin@math.columbia.edu}
\medskip
\address{Hui Yu, Department of Mathematics, Columbia University, 2990 Broadway, New York, NY} 
\email{huiyu@math.columbia.edu}

\begin{document}

\thispagestyle{empty}

\begin{minipage}{0.28\textwidth}
\begin{figure}[H]
\includegraphics[width=2.5cm,height=2.5cm,left]{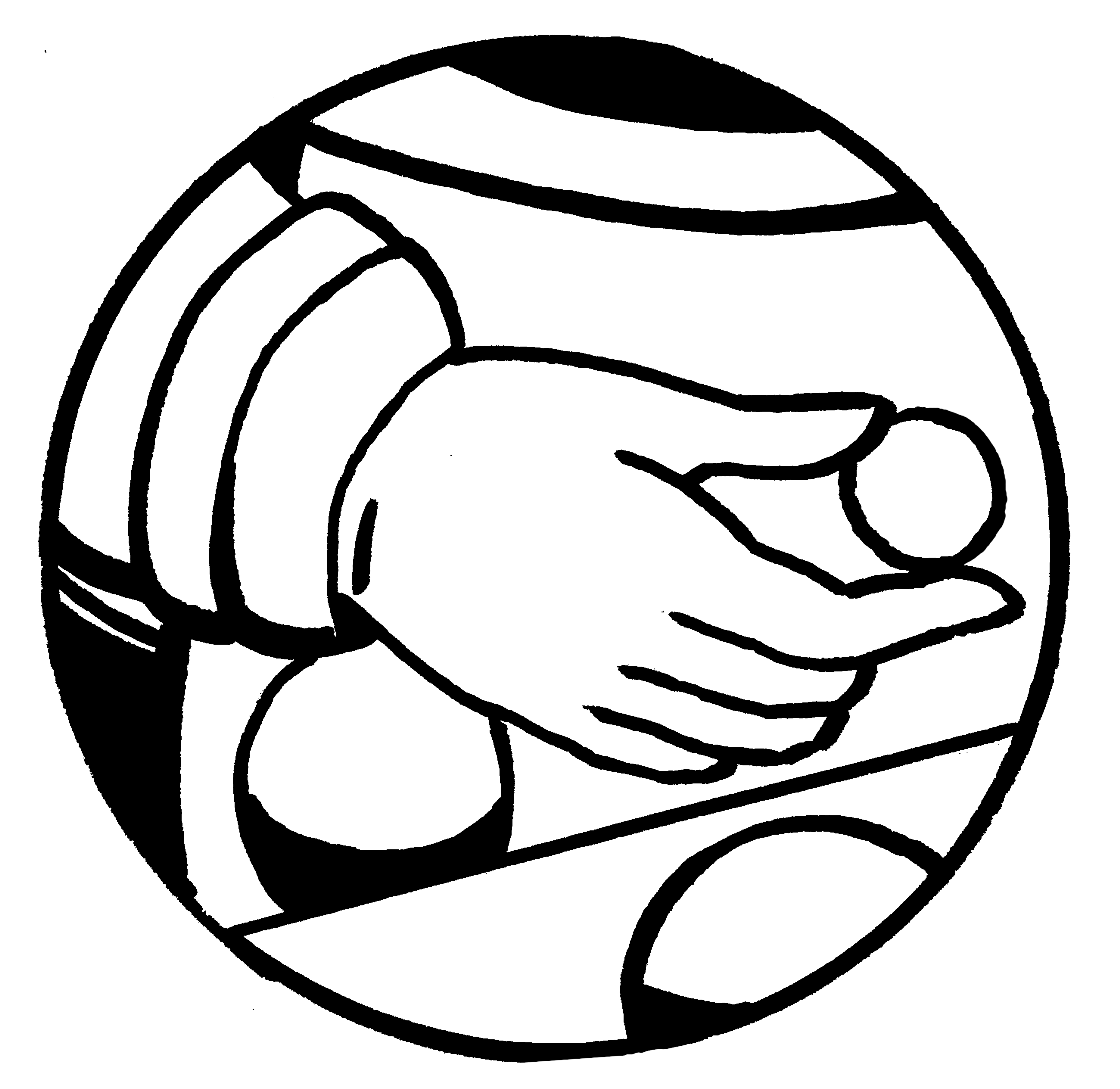}
\end{figure}
\end{minipage}
\begin{minipage}{0.7\textwidth} 
\begin{flushright}
Ars Inveniendi Analytica (2021), Paper No. 3, 49 pp.
\\
DOI 10.15781/ys6e-4d80
\end{flushright}
\end{minipage}

\ccnote

\vspace{1cm}


\begin{center}
\begin{huge}
\textit{Free boundary regularity}

\textit{  in the triple membrane problem}

\end{huge}
\end{center}

\vspace{1cm}


\begin{minipage}[t]{.28\textwidth}
\begin{center}
{\large{\bf{Ovidiu Savin}}} \\
\vskip0.15cm
\footnotesize{Columbia University}
\end{center}
\end{minipage}
\hfill
\noindent
\begin{minipage}[t]{.28\textwidth}
\begin{center}
{\large{\bf{Hui Yu}}} \\
\vskip0.15cm
\footnotesize{Columbia University}
\end{center}
\end{minipage}
\hfill
\noindent

\vspace{1cm}


\begin{center}
\noindent \em{Communicated by Guido De Philippis}
\end{center}
\vspace{1cm}


\noindent \textbf{Abstract.} \textit{We investigate the regularity of the free boundaries in the three elastic membranes problem.} 

\textit{We show that the two free boundaries corresponding to the coincidence regions between consecutive membranes are $C^{1,\log}$-hypersurfaces near a regular intersection point.  We also study two types of singular intersections. The first type of singular points are locally covered by a $C^{1,\alpha}$-hypersurface.  The second type of singular points stratify and each stratum is locally covered by a $C^1$-manifold. }
\vskip0.3cm

\vspace{0.5cm}


\section{Introduction}
For an integer $N\ge 2$, the $N$-membrane problem describes the shapes of $N$ elastic membranes under the action of forces. Mathematically, given a domain $\Omega\subset\R^d$ and  bounded functions $\{f_k\}_{k=1,2,\dots,N}$, we study the minimizer of the following functional \begin{equation}\label{FirstEquation}(u_1,u_2,\dots,u_N)\mapsto\int_{\Omega}\sum(\frac{1}{2}|\nabla u_k|^2+f_ku_k)\end{equation} over the class of  functions with prescribed data on $\partial\Omega$, and subject to the constraint \begin{equation}\label{SecondEquation}u_1\ge u_2\ge\dots\ge u_N \text{ in $\Omega$.}\end{equation} The function $f_k$ represents the force acting on the $k$th membrane, whose height is described by the unknown $u_k.$ Since the membranes cannot penetrate each other, the functions $\{u_k\}$ are well-ordered inside the domain. This leads to  the constraint \eqref{SecondEquation}. On the other hand, consecutive membranes can contact each other. Between the contact region $\{u_k=u_{k+1}\}$ and the non-contact region $\{u_k>u_{k+1}\}$, we have the \textit{free boundary} $\partial\{u_k>u_{k+1}\}.$ 

Existence and uniqueness of the minimizer in the multiple membrane problem were established  by Chipot and Vergara-Caffarelli \cite{CV}. They also proved that solutions are in $C^{1,\alpha}(\Omega)$ for all $\alpha\in(0,1)$. When the force terms $\{f_k\}$ are H\"older continuous, the authors recently obtained  the optimal $C^{1,1}$-regularity of solutions in Savin-Yu  \cite{SY}.   

The remaining questions that need to be addressed concern the regularity of the $N-1$ free boundaries $\partial\{u_k>u_{k+1}\}$ for  $k=1,2\dots, N-1$. To this end, it is natural to consider the case of \textit{constant} force terms that  satisfy a \textit{non-degeneracy condition} specific in obstacle-type problems $$f_1>f_2>\dots>f_N.$$

When $N=2$, there is only one free boundary $\partial\{u_1>u_2\}$, and the problem is equivalent to the classical obstacle problem for the difference $u_1-u_2$. In the non-contact region $\{u_1>u_2\}$, $\Delta(u_1-u_2)=f_1-f_2$ is constant.  This implies that $\partial\{u_1>u_2\}$ enjoys the same regularity as the free boundary in the obstacle problem which was extensively studied, see \cite{C,C2,W,M,CSV,FSe}.
In particular $\partial\{u_1>u_2\}$ is a smooth hypersurface outside a singular set of possible cusps. Similar results were proved for problems involving nonlinear operators by Silvestre \cite{Si},  and  even for problems involving operators of different orders 
 in Caffarelli-De Silva-Savin \cite{CDS}. 
 
With one more membrane, the situation changes drastically. 

When $N=3$, we have a coupled system of obstacle problems with interacting free boundaries, $\partial\{u_1>u_2\}$ and $\partial\{u_2>u_3\}$, which can cross each other. It can be viwed as a natural extension of the obstacle problem to the vector valued case.

To the knowledge of the authors, up to now very little is known about free boundary problems with interacting free boundaries, although these problems arise naturally in various contexts, see for instance Aleksanyan \cite{A}, Andersson-Shahgholian-Weiss \cite{ASW} and Lee-Park-Shahgholian \cite{LPS}.

It is instructive to look at the Euler-Lagrange equations  when $f_1=1,f_2=0,$ and $ f_3=-1$. For the regularity of $\partial\{u_1>u_2\}$, it is useful to write the equation for the difference $u_1-u_2:$ $$\Delta (u_1-u_2)=\mathcal{X}_{\{u_1>u_2>u_3\}}+\frac{3}{2}\mathcal{X}_{\{u_1>u_2=u_3\}}.$$ In the non-contact region $\{u_1>u_2\}$, the right-hand side jumps between $1$ and  $\frac{3}{2}$. This occurs when the two free boundaries, $\partial\{u_1>u_2\}$ and $\partial\{u_2>u_3\}$, cross each other.  When this happens, most of the known methods from the obstacle problem fail to apply.  As a result, very little is understood about the free boundaries when $N>2$, except that they are porous and have zero Lebesgue measure, see \cite{LR}.

In this work, we develop new techniques to deal with the system of interacting free boundaries. 
They apply to general H\"older continuous forcing terms that satisfy the non-degeneracy condition, however in order to focus on the main ideas, we assume throughout that $$f_1=1, f_2=0 \text{ and } f_3=-1.$$  
In this case, the average $(u_1+u_2+u_3)/3$ is harmonic.  Subtracting it from each $u_k$ does not affect the problem or the free boundaries. Hence we can assume $$u_1+u_2+u_3=0.$$ 
In a neighborhood of a point on $\partial\{u_1>u_2\}$ which does not intersect $\partial\{u_2>u_3\}$, the problem reduces back to the obstacle problem with constant right hand side for the difference $u_1-u_2$. Therefore in this neighborhood, $\partial \{u_1>u_2\}$ inherits the regularity properties of the free boundary in the classical obstacle problem. Thus it suffices to study what happens near points where the two free boundaries $\partial\{u_1>u_2\}$ and $\partial\{u_2>u_3\}$ intersect.

Suppose $x_0\in\partial\{u_1>u_2\}\cap\partial\{u_2>u_3\}$, and we define the rescaled solutions $$((u_1)_r,(u_2)_r,(u_3)_r)(x)=\frac{1}{r^2}(u_1,u_2,u_3)(x_0+rx).$$  Up to a subsequence of $r\to 0$, they converge to $2$-homogeneous solutions, see \cite{SY}.

 It is illustrative to look at four such blow-up profiles. See Figure \ref{OneDSolution}.
\begin{enumerate}
\item{The stable half-space solution: $$ u_1=\frac{1}{2}\max\{x_1, 0\}^2, u_2=0, u_3=-u_1.$$}
\item{The unstable half-space solution: \begin{align*}u_1&=\frac{1}{2}\min\{x_1,0\}^2+\frac{1}{4}\max\{x_1,0\}^2,\\ u_2&=-\frac{1}{4}\min\{x_1,0\}^2+\frac{1}{4}\max\{x_1,0\}^2,\\ u_3&=-\frac{1}{4}\min\{x_1,0\}^2-\frac{1}{2}\max\{x_1,0\}^2.\end{align*}}
\item{The hybrid solution:
\begin{align*}u_1&=\frac{1}{4}\max\{x_1, 0\}^2+\frac{1}{4}x\cdot Bx, \\u_2&=-\frac{1}{4}\max\{x_1, 0\}^2+\frac{1}{4}x\cdot Bx,\\ u_3&=-\frac{1}{2}x\cdot Bx,\end{align*}
where $B$ is a symmetric matrix satisfying $trace(B)=1$ and $3x\cdot Bx\ge x_1^2$; or
$$u_1=\frac{1}{2}x\cdot Bx, u_2=\frac{1}{4}\max\{x_1, 0\}^2-\frac{1}{4}x\cdot Bx, u_3=-u_1-u_2.$$
}
\item{The parabola solution: $$u_1=\frac{1}{2}x\cdot Ax, u_2=\frac{1}{2}x\cdot Bx, u_3=-u_1-u_2$$where $A,B$ are symmetric matrices with $trace(A)=1,$ $trace(B)=0$, $A\ge B$ and $A+2B\ge 0.$}
\end{enumerate}
In  \cite{SY} we showed that in the plane, up to a rotation,  these  profiles are the only $2$-homogeneous solutions. A similar classification holds for general $N$.

\begin{figure}[h]
\includegraphics[width=1\linewidth]{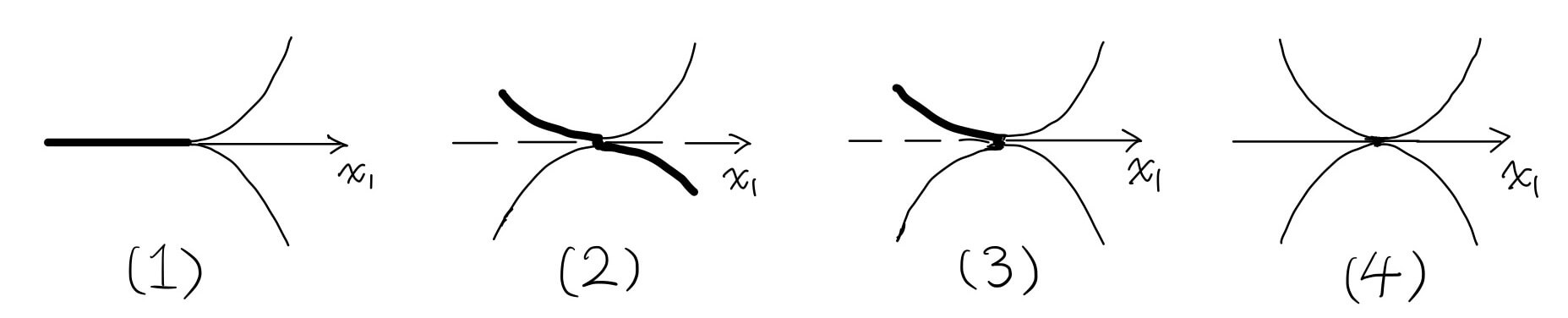}
\caption{Homogeneous solutions on $\R$.}
\label{OneDSolution}
\end{figure}

Given an intersection point $x_0\in\partial\{u_1>u_2\}\cap\partial\{u_2>u_3\}$, we say that $x_0$ is a \textit{regular point} if a subsequence of  rescalings converge to a (rotated) stable half-space solution. We call $x_0$ a \textit{singular point of type 1} if  a subsequence of rescalings converge to a (rotated) unstable half-space solution. Also, we say $x_0$ is a  \textit{singular point of type 2} if  a subsequence of rescalings converge to a parabola solution. The precise definitions are postponed to the next section. 

Around a point $x_0\in\partial\{u_1>u_2\}\cap\partial\{u_2>u_3\}$ where the rescalings converge to a hybrid solution, the behavior of the free boundaries is qualitatively different, and will be addressed in a future work.

For both types of half-space solutions, the contact sets are half spaces. It is intriguing that  the two free boundaries coincide in both cases. Heuristically, this says that the free boundaries intersect tangentially at points where the contact sets have positive density. 

To be precise, our result for regular intersection points is:

\begin{theorem}\label{MainResult2}
Suppose that $(u_1,u_2,u_3)$ is a solution to the $3$-membrane problem in $\Omega$. Let $\mathbf{Reg}$ denote the collection of regular points. 

Then for $x_0\in\mathbf{Reg}\cap\Omega$, there is $r>0$ such that $$\partial\{u_1>u_2\}\cap\partial\{u_2>u_3\}\cap B_r(x_0)=\mathbf{Reg}\cap B_r(x_0),$$ and both $\partial\{u_1>u_2\}$ and $\partial\{u_2>u_3\}$ are $C^{1,\log}$-hypersurfaces in $B_r(x_0)$, intersecting tangentially. 
\end{theorem} 
\begin{remark}We remark that the $C^{1,\log}$-regularity is optimal, and it occurs at regular intersection points under small generic perturbations, see Proposition \ref{GenericOptimal}.  The generic condition, that is used here and later in Remark \ref{SecondRemark}, is inspired by the work of Colding-Minicozzi on mean curvature flows \cite{CM}.
\end{remark}

\begin{remark}

Although our approach follows in the spirit of the improvement-of-flatness technique, we point out that a standard application of this technique does not work in our problem. Traditionally, this technique is only applicable to problems where the free boundary is at least $C^{1,\alpha}$, which allows a linearization of the problem. In our problem, however, the free boundary is only $C^{1,\log}$, and a direct linearization is not possible. 

Instead, we establish a dichotomy as in Proposition \ref{IOF1}, which might be the most novel contribution of this work. The iteration of such dichotomy naturally leads to $C^{1,\log}$-regularity when classical techniques do not apply. This same strategy has recently been applied to the thin obstacle problem in Savin-Yu \cite{SY4}.
\end{remark}

Our result for singular points of type 1 is:

 \begin{theorem}\label{MainResult3}
Suppose that $(u_1,u_2,u_3)$ is a solution to the $3$-membrane problem in $\Omega$. Let $\mathbf{Sing}_1$ denote the collection of singular points of type 1. 

Then $\mathbf{Sing}_1\cap\Omega$ is locally covered by a $C^{1,\alpha}$-hypersurface. 
\end{theorem} 
\begin{remark}\label{SecondRemark}Singular points of type 1 are not stable. Under generic local perturbations, they are removed from $\partial\{u_1>u_2\}\cap\partial\{u_2>u_3\}$, see Remark \ref{SOneNotStable}.\end{remark}

For parabola solutions, the contact sets $\{u_1=u_2\}$ and $\{u_2=u_3\}$ are of lower dimensions. This tangential contact implies that the solution, before blowing up, is $C^2$ at a singular point of type 2. The situation is reminiscent to that of a singular point in the obstacle problem.  

To be precise, our result for singular points of type 2 is the following:
\begin{theorem}\label{MainResult1}
Suppose that $(u_1,u_2,u_3)$ is a solution to the $3$-membrane problem in $\Omega$. Let $\mathbf{Sing}_2$ denote the collection of singular points of type 2. Then $$\mathbf{Sing}_2\cap\Omega=\cup_{k=0,1,\dots, d-1}\Sigma^k,$$ where $\Sigma^0$ consists of isolated points, and  $\Sigma^k$ is locally covered by a $C^1$-manifold of dimension $k$ for each $k=1,\dots, d-1$.
\end{theorem} Recall that $d$ in the theorem above is the dimension of the ambient space. 

It is interesting to note that Theorem \ref{MainResult1} holds for a general number of membranes $N$. The counterparts of Theorem \ref{MainResult2} and Theorem \ref{MainResult3} when $N\ge 4$, however,  seem to be out of reach at the moment. The main difficulty is that around points in $\mathbf{Reg}$ and $\mathbf{Sing}_1$, the behavior of the solutions are not described by the corresponding blow-up limits. For instance, at a regular point, all blow-up solutions are rotations of the first profile in Figure \ref{OneDSolution}. On the other hand, for a typical solution (before blowing up), the two free boundaries separate and the solutions fail to be one-dimensional. This break of symmetry lies behind several important open problems in free boundary problems as well as geometric analysis \cite{DSV}. 
When $N=3$, we overcome this challenge with a hidden comparison principle in the system. See Proposition \ref{Comparison}.

This paper is organized as follows. In the next section, we gather several  definitions and previous results from Savin-Yu \cite{SY}. In Section 3, we reformulate the $3$-membrane problem as a coupled system of obstacle problems. In Sections 4 and 5, we work with this reformulation and give two improvement of flatness results. These are the heart of the paper. In Sections 6 and 7, we prove Theorem \ref{MainResult2} and Theorem \ref{MainResult3}, respectively. In  these two sections, we also point out the optimality of the results as well as what happens under generic perturbations. In Section 8, we give the proof of Theorem \ref{MainResult1}. 

\medskip

\textbf{Acknowledgement}:
O.~S.~is supported by  NSF grant DMS-1500438. 
H.~Y.~is supported by NSF grant DMS-1954363.

\section{Preliminaries}
In this section we collect some preliminary materials. Most of the results here can be found in Savin-Yu \cite{SY}.

We begin with the definition of a solution to the $3$-membrane problem: 
\begin{definition}\label{Solution}
Let $\Omega$ be a domain in $\R^d$. 

A triplet of continuous functions  on $\overline{\Omega}$, $(u_1,u_2,u_3)$, is called  \textit{a solution to the $3$-membrane problem in $\Omega$} if 
\begin{enumerate}\item{$u_1+u_2+u_3=0 \text{ and }  u_1\ge u_2\ge u_3 \text{ in $\Omega$,}$ and } \item{the following equations are satisfied
$$\begin{cases}\Delta u_1=&\mathcal{X}_{\{u_1>u_2\}}+\frac{1}{2}\mathcal{X}_{\{u_1=u_2>u_3\}},\\
\Delta u_2=&\frac{1}{2}\mathcal{X}_{\{u_1=u_2>u_3\}}-\frac{1}{2}\mathcal{X}_{\{u_1>u_2=u_3\}},\\
\Delta u_3=&-\mathcal{X}_{\{u_2>u_3\}}-\frac{1}{2}\mathcal{X}_{\{u_1>u_2=u_3\}}.
\end{cases}$$}\end{enumerate}
\end{definition}  
This is the system of Euler-Lagrange equations for a minimizer of \eqref{FirstEquation} under the constraint in \eqref{SecondEquation}, when $N=3$ and $f_1=1, f_2=0 ,f_3=-1.$

To simplify notations, we denote the two \textit{free boundaries} by $$\Gamma_1=\partial\{u_1>u_2\}\cap\Omega \text{ and }\Gamma_2=\partial\{u_2>u_3\}\cap\Omega.$$The main question we study in this paper  is the regularity of $\Gamma_1$ and $\Gamma_2.$

Around points on $\Gamma_1\cap\{u_2>u_3\}$ and $\Gamma_1\cap \mathrm{Int}\{u_2=u_3\},$  the problem reduces to the $2$-membrane problem for $(u_1,u_2)$, for which the regularity has been fully addressed. The same happens for points on  $\Gamma_2\cap\{u_1>u_2\}$ and $\Gamma_2\cap \mathrm{Int}\{u_1=u_2\}.$ As a result, it suffices to study the regularity of $\Gamma_k$  $(k=1,2)$ near free boundary points \textit{with the highest multiplicity}, namely, points on $\Gamma_1\cap\Gamma_2.$

Around the free boundaries, we have the following:
\begin{proposition}\label{QuadraticGrowth}
Let $\Trip$ be a solution to the $3$-membrane problem in $B_1$ with $0\in\Gamma_{1}$. 

Then there is a dimensional constant $0<C<\infty$ such that $$\frac{1}{4d}r^2\le\sup_{B_r}(u_1-u_2)\le Cr^2 \text{ for $r\in(0,1)$.}$$

Similar estimates hold for $(u_2-u_3)$ if $0\in\Gamma_2$.
\end{proposition}  
Recall that $d$ is the dimension of the ambient space. 

As a consequence, we have the optimal regularity of solutions:
\begin{theorem}\label{OptimalRegularity}
Let $\Trip$ be a solution to the $3$-membrane problem in $B_1$ with $0\in\Gamma_{1}\cap\Gamma_2$. 

Then $\sum\|D^2u_k\|\le C$ in $B_{1/2}$ for a dimensional constant $C$.
\end{theorem}

This gives compactness of the family of rescaled solutions. To be precise, for $x_0\in\Gamma_1\cap\Gamma_2$, we define $$(u_k)_{x_0,r}=\frac{1}{r^2}u_k(x_0+rx) \text{ for $k=1,2,3$}.$$ As $r\to0$, these functions are locally uniformly  bounded in $C^{1,1}$. Consequently, there are functions $(u_k)_{x_0}\in C^{1,1}_{loc}(\R^d)$ such that, up to a subsequence, $$(u_k)_{x_0,r}\to (u_k)_{x_0} \text{ locally uniformly.}$$ 

The triplet $((u_k)_{x_0})_{k=1,2,3}$ is called a \textit{blow-up profile} at $x_0$. 

This is a slight abuse of notation. At this stage, we do not have uniqueness of blow-ups.  This blow-up profile  not only depends on the point $x_0$, but could also depend on the particular subsequence of $r\to0.$ An important result of this paper is that for the three types of free boundary points in Definition \ref{FreeBoundaryPoints}, blow-ups are indeed unique.

The blow-up profile $((u_k)_{x_0})$ solves the $3$-membrane problem in $\R^d$. The origin is a free boundary point with the highest multiplicity, that is,  $$0\in\partial\{(u_1)_{x_0}>(u_2)_{x_0}\}\cap\partial\{(u_2)_{x_0}>(u_3)_{x_0}\}.$$

To study these blow-up profiles, we use  a monotonicity formula inspired by the Weiss energy \cite{W}. This monotonicity formula holds for general $N\ge 2$. In this paper, we only need the special case when $N=3$. 

For a point $x_0\in\Gamma_1\cap\Gamma_2$ and small $r>0$, the functional is defined as 
\begin{equation}\label{WeissEnergy}\begin{split}W((u_k),x_0,r)=&\frac{1}{r^{d+2}}\int_{B_r(x_0)}\frac{1}{2}\sum|\nabla u_k|^2+u_1-u_3\\-&\frac{1}{r^{d+3}}\int_{\partial B_r(x_0)}\sum u_k^2.\end{split}\end{equation} This is monotone:
\begin{theorem}\label{Monotonicity}
Let $\Trip$ be a solution to the $3$-membrane problem in $B_1$ with $0\in\Gamma_1\cap\Gamma_2.$  Then $W((u_k),0,r)$ is a non-decreasing function in $r$.   

Moreover, if $W((u_k),0,r)$ is constant in $r$, then $(u_1,u_2,u_3)$ is $2$-homogeneous in $B_1.$
\end{theorem} 

This gives strong restrictions on blow-up profiles:
\begin{corollary}
Let $\Trip$ be a solution to the $3$-membrane problem in $\Omega$ with $x_0\in\Gamma_{1}\cap\Gamma_2.$ Suppose $((u_1)_{x_0},(u_2)_{x_0},(u_3)_{x_0})$ is a blow-up profile at $x_0$. 

Then for $k=1,2,3,$ $(u_k)_{x_0}$ is a $2$-homogeneous function in $\R^d.$
\end{corollary}

In two dimensions, homogeneous solutions have been completely classified, even for general $N$ \cite{SY}. 

In what follows, we use the following standard notation:
\begin{notation}
We denote by $\Sph$ the set of unit vectors in $\R^d$.  

The standard basis for $\R^d$ is denoted by $\{e^k\}_{k=1,2,\dots,d}$. The coordinate function in the direction of $e^k$ is denoted by $x_k.$  
\end{notation}

With these notations, we extend the homogeneous solutions from 2D to general dimensions. Under the assumption $\sum u_k=0,$ it suffices to define $u_1$ and $u_3$.

\begin{definition}\label{SHS}
For  $e\in\Sph$, the \textit{stable half-space solution in direction $e$} is $$ u_1^{0,e}=\frac{1}{2}\max\{x\cdot e, 0\}^2, u_3^{0,e}=-u_1^{0,e}.$$ The class of stable half-space solutions is denoted by $\mathcal{SH}$, that is, 
$$\mathcal{SH}=\{(u_k^{0,e}):e\in\Sph\}.$$
\end{definition} 

\begin{definition}\label{UHS}
For  $e\in\Sph$, the \textit{unstable half-space solution in direction $e$} is  \begin{align*}u_1^{1,e}&=\frac{1}{2}\min\{x\cdot e,0\}^2+\frac{1}{4}\max\{x\cdot e,0\}^2,\\ u_3^{1,e}&=-\frac{1}{4}\min\{x\cdot e,0\}^2-\frac{1}{2}\max\{x\cdot e,0\}^2.\end{align*}The class of unstable half-space solutions is denoted by $\mathcal{UH}$, that is, 
$$\mathcal{UH}=\{(u_k^{1,e}):e\in\Sph\}.$$
\end{definition} 

\begin{definition}\label{HybridSolution}
For a symmetric matrix $B$ and a unit vector $e\in\Sph$ satisfying $$trace(B)=1\text{ and } 3B-e\otimes e\ge 0,$$ the \textit{hybrid solution with direction $e$ and coefficient matrix $B$} is
$$u_1^{(e,B)}=\frac{1}{4}\max\{x\cdot e,0\}^2+\frac{1}{4}x\cdot Bx, \quad u_3^{(e,B)}=-\frac 12x\cdot Bx.$$
Symmetrically, the \textit{hybrid solution with  coefficient matrix $B$ and direction $e$} is
$$u_1^{(B,e)}=\frac{1}{2}x\cdot Bx, \quad u_3^{(e,B)}=-\frac{1}{4}\max\{x\cdot e,0\}^2-\frac 14x\cdot Bx.$$

\end{definition}

\begin{definition}\label{ParabolaSolution}
For symmetric matrices $A$ and $B$ with $trace(A)=1$, $trace(B)=-1$, and $2A+B\ge 0\ge A+2B$, the \textit{parabola solution with coefficient matrices $(A,B)$} is 
$$u_1^{(A,B)}=\frac{1}{2}x\cdot Ax, u_3^{(A,B)}=\frac{1}{2}x\cdot Bx.$$
The class of parabola solutions is denoted by $\mathcal{P}$, that is, 
$$\mathcal{P}=\{(u_k^{(A,B)}): trace(A)=1 ,  trace(B)=-1 , 2A+B\ge 0\ge A+2B\}.$$
\end{definition}

One consequence of Theorem \ref{Monotonicity} is that there is a well-defined function for $x_0\in\Gamma_{1}\cap\Gamma_2$ by \begin{equation}\label{LimitOfWeiss}W((u_k),x_0):=\lim_{r\to 0}W((u_k),x_0,r)=W((u_k)_{x_0},0,1),\end{equation} where $(u_k)_{x_0}$ is a blow-up profile at $x_0.$

For the four types of solutions above, we have \begin{equation}\label{FirstEnergy}W((u_k^{0,e}),0)=W_{0}, W((u_k^{1,e}),0)=W_{1}, W((u_k^{(e,B)}),0)=W_{2} \text{ and } W((u_k)^{A,B})=W_3,\end{equation} where each $W_k$ is a positive dimensional constant. They satisfy $$W_1=\frac{3}{2}W_0, W_2=\frac{7}{4}W_0  \text{ and } W_3=2W_0.$$
Heuristically, this implies that among the four types of solutions, the stable half-space solution is the most stable, and the parabola solution is the least stable. 

This motivates the following definition:

\begin{definition}\label{FreeBoundaryPoints}
Suppose $x_0\in\Gamma_{1}\cap\Gamma_2$.  We define the following four classes of free boundary points with the highest multiplicity:
\begin{enumerate}
\item{We call $x_0$ a \textit{regular point} if there is a blow-up profile in $\mathcal{SH}.$ 

The collection of regular points is denoted by $\Reg$.}
\item{We call $x_0$ a \textit{singular point of type 1} if there is a blow-up profile in $\mathcal{UH}.$ 

The collection of singular points of type 1 is denoted by $\SOne$.}
\item{We call $x_0$ a \textit{hybrid point} if there is a blow-up profile given by a hybrid solution as in Definition \ref{HybridSolution}.
}

\item{We call $x_0$ a \textit{singular point of type 2} if there is a blow-up profile in  $\mathcal{P}.$ 

The collection of singular points of type 2 is denoted by $\STwo$.}

\end{enumerate}
\end{definition}

Thanks  to the classification of homogeneous solutions in 2D, these four classes form a partition of $\Gamma_{1}\cap\Gamma_2.$  For general dimensions, the comparison of $W_0$, $W_1$, $W_2$ and $W_3$ implies that the four classes are mutually disjoint. However, it is not clear whether they exhaust the entire $\Gamma_1\cap\Gamma_2.$ 

As mentioned in Introduction, we focus on regular points and  two types of singular points in the remaining part of this paper. Free boundary regularity around hybrid points will be addressed in  a future work.

\section{A system of obstacle problems}
In this section, we reformulate the $3$-membrane problem as a  coupled system of two obstacle problems.  This system enjoys a more transparent comparison principle. 

Suppose that $\Trip$ is a solution to the $3$-membrane problem in $\Omega$. The first  condition in Definition \ref{Solution} leads to   $$u_1\ge -\frac{1}{2}u_3 \text{ in $\Omega$}.$$  The second condition gives $$\Delta u_1\le 1 \text{ in $\Omega$,}$$and$$\Delta u_1=1 \text{ in $\{u_1>-\frac{1}{2}u_3\}$}.$$ That is, $u_1$ solves the obstacle problem with the \textit{unknown} obstacle $-\frac{1}{2}u_3$.   A similar argument applies to $-u_3$, which solves the obstacle problem with $\frac{1}{2}u_1$ as the obstacle. 
 
With this observation, we recast the $3$-membrane problem as a coupled  system of two obstacle problems:
\begin{definition}\label{ReSol}
Let $\Omega$ be a domain in $\R^d$. Suppose $(u,w)$ is a pair of continuous functions on $\overline{\Omega}$ satisfying $$u\ge0,\text{ } w\ge 0, \text{ } u\ge\frac{1}{2}w \text{ and } w\ge\frac{1}{2}u.$$
We say that the pair $(u,w)$ is a \textit{subsolution} in $\Omega$ (to the system of obstacle problems), and write $$(u,w)\in\underline{\mathcal{A}}(\Omega),$$ if $$\Delta u\ge\mathcal{X}_{\{u>\frac{1}{2}w\}} \text{ and } \Delta w\ge\mathcal{X}_{\{w>\frac{1}{2}u\}} \text{ in $\Omega.$}$$
We say that the pair $(u,w)$ is a \textit{super solution} in $\Omega$, and write $$(u,w)\in\overline{\mathcal{A}}(\Omega),$$ if $$\Delta u\le 1 \text{ and } \Delta w\le 1 \text{ in $\Omega.$}$$
The pair $(u,w)$ is called a \textit{solution} in $\Omega$ if $(u,w)\in\Subsol(\Omega)\cap\Supsol(\Omega).$ In this case, we write $$(u,w)\in\mathcal{A}(\Omega).$$
\end{definition} 

\begin{remark}\label{EquivalenceBetweenProblems}This problem is equivalent to the $3$-membrane problem,  in the sense that $(u,w)\in\mathcal{A}(\Omega)$ if and only if the triplet $(u, -u+w, -w)$ solves the $3$-membrane problem as in Definition \ref{Solution}. In particular, there are  two free boundaries in the reformulated problem, namely, \begin{equation*}\label{Gu}\Gamma_u:=\partial\{u>\frac{1}{2}w\}\cap\Omega, \text{ and  } \Gamma_w:=\partial\{w>\frac{1}{2}u\}\cap\Omega.\end{equation*} \end{remark}

With this equivalence, we have the following two results in the spirit of Proposition \ref{QuadraticGrowth} and Theorem \ref{OptimalRegularity}:

\begin{proposition}\label{QuadGrowthRef}
Suppose $(u,w)\in\mathcal{A}(B_1)$ with  $0\in\Gamma_{u}$.  Then there is a dimensional constant $0<C<\infty$ such that  $$\frac{1}{4d}r^2\le\sup_{B_r}(u-\frac{1}{2}w)\le Cr^2 \text{ for $r\in (0,1)$.}$$

Similar estimates hold for $(w-\frac{1}{2}u)$ if $0\in\Gamma_w$.
\end{proposition}  

\begin{theorem}
Suppose $(u,w)\in\mathcal{A}(B_1)$ with  $0\in\Gamma_{u}\cap\Gamma_w$.

Then $\|D^2u\|+\|D^2w\|\le C$ in $B_{1/2}$ for a dimensional constant $C$.
\end{theorem} 

As a direct consequence of Proposition \ref{QuadGrowthRef}, solutions enjoy the following non-degeneracy properties:
\begin{lemma}\label{NonDegeneracy}
Suppose $(u,w)\in\mathcal{A}(B_r)$. 

If $u-\frac{1}{2}w\le\frac{1}{4d}r^2$ along $\partial B_r$, then $u(0)=\frac{1}{2}w(0).$

If $u\le\frac{1}{4d}r^2$ along $\partial B_r$, then $u(0)=0.$ 

If $w-\frac{1}{2}u\le\frac{1}{4d}r^2$ along $\partial B_r$, then $w(0)=\frac{1}{2}u(0).$

If $w\le\frac{1}{4d}r^2$ along $\partial B_r$, then $w(0)=0.$ 
\end{lemma}

We  need to compare pairs of functions. To simplify notations,
we write \begin{equation}\label{ComparisonBetweenPairs}(u,w)\ge(u',w')\end{equation}   if $u\ge w$ and $w\ge w'$. 

We have the comparison principle between a subsolution and a super solution:
\begin{proposition}\label{Comparison}
Let $\Omega$ be  bounded.

Suppose $(u,w)\in\Subsol(\Omega)$ and $(u',w')\in\Supsol(\Omega)$. 

If $(u,w)\le(u',w')$ along $\partial\Omega,$ then $(u,w)\le(u',w')$ in $\Omega.$
\end{proposition} 

\begin{proof}
Define $t_0=\inf\{t\in\R:(u'+t,w'+t)\ge (u,w) \text{ in $\overline{\Omega}$}\}.$ It suffices to show that $t_0\le 0.$

Suppose, on the contrary, that $t_0>0$. 

By the definition of $t_0$, there is a point $x_0\in\overline{\Omega}$ such that $$\text{either } u(x_0)=u'(x_0)+t_0\text{ or }w(x_0)=w'(x_0)+t_0.$$ We only deal with the first case. The argument for the other is similar. 

With $(u,w)\le(u',w')$ along $\partial\Omega$ and $t_0>0$, this point $x_0$ is in the interior of $\Omega.$ 

With $t_0>0$,  we have $$u(x_0)=u'(x_0)+t_0>\frac{1}{2}w'(x_0)+\frac{1}{2}t_0\ge\frac{1}{2}w(x_0).$$ By continuity, the comparison $u>\frac{1}{2}w$  holds in an entire neighborhood of $x_0$, say, $\mathcal{N}.$ 

Inside $\mathcal{N}$, we have $$\Delta u\ge 1\ge \Delta (u'+t_0).$$ Also $u\le u'+t_0$ and $u(x_0)=u'(x_0)+t_0.$ The strong comparison principle implies $$u=u'+t_0 \text{ inside $\mathcal{N}.$}$$
Consequently, we can replace $x_0$ by any point $y_0\in\mathcal{N}$, and use the same argument to get $u=u'+t_0$ in a neighborhood of $y_0.$

This implies $u=u'+t_0$ in the entire $\Omega$. Continuity forces $u=u'+t_0$ along $\partial\Omega$, contradicting the comparison along $\partial\Omega$ since $t_0>0.$
\end{proof} 

A more useful version is as follows:
\begin{lemma}\label{FirstComparison}
Suppose $(u,w)\in\Subsol(B_2)$ and $(u',w')\in\Supsol(B_2)$ with $u'\le 1$ along $\partial B_2.$

If, for some $\eps\in (0,\frac{1}{4d})$, we have $$(u',w')\ge (u+10d \, \eps\mathcal{X}_{\{u'>\frac{1}{4d}\}}-\eps,w+10d \, \eps\mathcal{X}_{\{w'>\frac{1}{4d}\}}-\eps) \text{ along $\partial B_2$,}$$
then $(u',w')\ge (u,w)$ in $B_1.$
\end{lemma} 
Recall that $d$ is the dimension of the ambient space. 

\begin{proof}
For each $x_0\in B_1$, define $$\varphi(x)=(1-4d\eps)u'(x)+2\eps|x-x_0|^2$$ and $$\psi(x)=(1-4d\eps)w'(x)+2\eps|x-x_0|^2.$$
Then $\varphi$ and $\psi$ are both non-negative and satisfy $$2\varphi=(1-4d\eps)2u'(x)+4\eps|x-x_0|^2\ge(1-4d\eps)w'(x)+4\eps|x-x_0|^2\ge\psi.$$Similarly, we have $2\psi\ge\varphi.$  

Moreover, with $\Delta u'\le 1$, we have $$\Delta\varphi=(1-4d\eps)\Delta u'+4d\eps\le 1.$$ Similarly, we have $\Delta\psi\le 1.$

Therefore, $(\varphi,\psi)\in\Supsol(B_2).$

We now compare $(\varphi,\psi)$ with $(u,v)$ along $\partial B_2.$ Note that here $|x-x_0|\ge 1.$

Along $\partial B_2\cap\{u'\le\frac{1}{4d}\},$ we have $u'\ge u-\eps$. Consequently, $$\varphi\ge u'-4d\eps u'+2\eps\ge u'-\eps+2\eps\ge u.$$

Along $\partial B_2\cap\{u'>\frac{1}{4d}\}$, we have $1\ge u'\ge u+10d\eps-\eps$. As a result, 
$$\varphi\ge u'-4d\eps+2\eps\ge u+10d\eps-\eps-4d\eps+2\eps\ge u. $$
To conclude, $\varphi\ge u$ along $\partial B_2.$ A similar argument gives $\psi\ge w$ along $\partial B_2.$

Now  Proposition \ref{Comparison} gives $(\varphi, \psi)\ge(u,w) \text{ in $B_2$.}$ At the point $x_0\in B_1$, this implies $(u',w')(x_0)\ge (u,w)(x_0).$\end{proof} 

We also have the symmetric comparison, which follows from similar arguments:
\begin{lemma}\label{SecondComparison}
Suppose $(u,w)\in\Supsol(B_2)$ and $(u',w')\in\Subsol(B_2)$ with $u\le 1$ along $\partial B_2.$

If, for some $\eps\in (0,\frac{1}{8d})$, we have $$(u',w')\le (u-10d\, \eps\mathcal{X}_{\{u'>\frac{1}{8d}\}}+\eps,w-10d\, \eps\mathcal{X}_{\{w'>\frac{1}{8d}\}}+\eps) \text{ along $\partial B_2$,}$$
then $(u',w')\le (u,w)$ in $B_1.$
\end{lemma} 
\section{Improvement of flatness: Case 1}
In this section and the next, we prove two improvement-of-flatness type results for the system of obstacle problems. We give these results in two cases that are relevant to the free boundary regularity as in Theorem \ref{MainResult2} and Theorem \ref{MainResult3}.

In this section, we study the case related to regular points in the $3$-membrane problem. By Definition \ref{FreeBoundaryPoints}, around such points the solution is approximated  by stable half-space solutions. For our argument, we need to include  rotations and translations of such profiles, that is, functions of the form 
\begin{equation}\label{Pabab1}
\Pabab=\frac{1}{2}\max\{x\cdot\alpha-a,0\}^2 \text{ and } \Qabab=\frac{1}{2}\max\{x\cdot\beta-b,0\}^2,
\end{equation}  where $\alpha,\beta\in\Sph$ and $a,b\in\R.$

We often write $\Pab$ and $\Qab$ or even just $P$ and $Q$  instead of $\Pabab$ and $\Qabab$ for these profiles.

In terms of the system of obstacle problems as in Definition \ref{ReSol}, we work with the following class of solutions:
\begin{definition}\label{RegPointSol}
For  $\alpha,\beta\in\Sph$, $a,b\in\R$ and $\eps>0$, we say that $$(u,w)\in\mathcal{R}(\alpha,\beta;a,b;\eps) \text{ in $B_r$}$$ if $$(u,w)\in\mathcal{A}(B_r) \text{ with } 0\in\Gamma_u,$$ and $$|u-\Pabab|<\eps r^2 \text{ and } |w-\Qabab|<\eps  r^2\text{ in $B_r.$}$$
\end{definition} 

We often write $$(u,w)\in\mathcal{R}(\alpha,\beta;\eps) \text{ in $B_r$}$$when there is no need to emphasize  $a$ and $b$.

Up to a rotation, it suffices to study the case when $\alpha$ and $\beta$ satisfy\begin{equation}\label{Symmetric}
\alpha_1=\beta_1>0, \alpha_2=-\beta_2\ge0 \text{ and } \alpha_k=\beta_k=0 \text{ for $k\ge3.$ }
\end{equation} 

Lemma \ref{NonDegeneracy} leads to bounds on the parameters:
\begin{lemma}\label{Localization1}
Suppose $(u,w)\in\mathcal{R}(\alpha,\beta;a,b;\eps)$ in $B_1$ with a small $\eps$. 

Then $$|a|,|b|, |\alpha-\beta|<C\eps^{1/2}$$ for a dimensional constant $C$. 
\end{lemma} 
\begin{proof}
\textit{Step 1: Localizing $a$ and $b$.}

Suppose $a>0$, then $B_a\subset\{x\cdot\alpha-a<0\}$. As a result, we have  $$u\le\eps \text{ in $B_a.$}$$ If $a>\sqrt{4d\eps}$, then Lemma \ref{NonDegeneracy} implies $u=0$ in a neighborhood of $0$, contradicting $0\in\Gamma_u.$ 

Therefore, $a\le\sqrt{4d\eps}$. Similarly, $b\le\sqrt{4d\eps}.$

On the other hand, with $u(0)=\frac{1}{2}w(0),$ we have \begin{equation*}-\frac{3}{2}\eps\le\frac{1}{2}\max\{-a,0\}^2-\frac{1}{4}\max\{-b,0\}^2\le\frac{3}{2}\eps.\end{equation*}
This implies that  $a, b\ge-C\eps^{1/2}$ for a dimensional $C$. 

\textit{Step 2: Localizing $(\alpha-\beta)$.}

The condition $u\ge\frac{1}{2}w$ at $-e^2$ implies
$$\frac{1}{2}\max\{-\alpha_2-a,0\}^2+\eps\ge\frac{1}{4}\max\{\alpha_2-b,0\}^2-\frac{1}{2}\eps.$$
With $\alpha_2\ge 0$ and $|a|,|b|\le C\eps^{1/2}$ from Step 1, this implies $\alpha_2\le C\eps^{1/2}.$
\end{proof}

The main result in this section is the following:
\begin{proposition}[Improvement of flatness: Case 1]\label{IOF1}
There are  small positive constants $\delta, \eps_d$ and $\rho_k$  $(k=0,1,2),$ and  large constants $M$ and $C$, depending only on the dimension,  such that the following holds:

Suppose $$(u,w)\in\mathcal{R}(\alpha,\beta;a,b;\eps)\text{ in $B_1$}$$ with $$|\alpha-\beta|<2\delta\eps^{1/2}$$ for some $\eps<\eps_d.$ Then we have two alternatives:

1) If $|a-b|>\rho_0|\alpha-\beta|+M\rho_0\eps^{3/4}$, then $\Gamma_u\cap\Gamma_w\cap B_{\rho_0}=\emptyset$, and  $\Gamma_u$ is a $C^{1,\gamma}$-hypersurface in $B_{\rho_0}$ with $C^{1,\gamma}$-norm  bounded by $C\eps$;

2) Otherwise, there are $\alpha',\beta'\in\Sph$ with $|\alpha'-\alpha|+|\beta'-\beta|<C\eps$ such that $$(u,w)\in\mathcal{R}(\alpha',\beta';\eps/2) \text{ in $B_{\rho_1}$.}$$  

Moreover, if $|\alpha-\beta|>\delta\eps^{1/2}$,  then  there are $\alpha'',\beta''\in\Sph$ with $|\alpha''-\alpha|+|\beta''-\beta|<C\eps$ such that $$(u,w)\in\mathcal{R}(\alpha'',\beta'';\frac{1}{32}\eps) \text{ in $B_{\rho_2}$}$$ and $$|\alpha''-\beta''|<|\alpha-\beta|-\eps.$$ 
\end{proposition} 

Let $\ell_1$ and $\ell_2$ denote the hyperplanes $\{x\cdot\alpha=a\}$ and $\{x\cdot\beta=b\}$, respectively. When the angle between $\ell_1$ and $\ell_2$ is small, this proposition deals with the case when $u$ and $w$ are approximated by half-space profiles with $\ell_1$ and $\ell_2$ as free boundaries. 

In the first alternative, $\ell_1$ and $\ell_2$ are well-separated in $B_{2\rho_0}$. In this case, the two free boundaries $\Gamma_u$ and $\Gamma_w$ decouple. The regularity of $\Gamma_u$ follows from the theory of the obstacle problem. 

In the second alternative, we improve the approximation of $(u,w)$ in $B_{\rho_1}$. An iteration of this improvement leads to regularity of $\Gamma_u$. To iterate, however,  the angle between hyperplanes needs to stay small. 

This issue becomes urgent once $|\alpha-\beta|$ reaches the critical level $\delta\eps^{1/2}$.  In this case, instead of $\rho_1$, we go to a much smaller scale $\rho_2$. At this scale, the angle decreases by a definite amount $\eps$.  After $k$ iterations, the angle is less than $(2\delta\eps^{1/2}-k\eps)$. Consequently, within at most $\eps^{-1/2}$ steps, the angle becomes subcritical. 

We give our proof of Proposition \ref{IOF1} in the following subsections. 

\subsection{Approximate solutions}In this subsection, we build approximate solutions based on half-space profiles. They play a crucial role in our analysis and are used in estimating the behavior of the solutions near the free boundaries. We first give some heuristics by looking at the problem in one dimension. 

\begin{eg}\label{1DEx}
Suppose on $\R$, we are given two half-space profiles $$P=\frac{1}{2}\max\{x-a,0\}^2 \text{ and }Q=\frac{1}{2}\max\{x-b,0\}^2$$ with small $a< b$. Our goal is to find an actual solution, say, $(\Phi,\Psi),$ that best approximates $(P,Q)$.

It is natural to take  $\Phi=P.$  

But $\Psi\ge\frac{1}{2}\Phi$ forces $\Psi\ge\frac{1}{4}(x-a)^2$ for $x\ge a.$ Meanwhile, $Q=0$ on $(a,b)$. To minimize the error, we take $\Phi=\frac{1}{4}(x-a)^2$ on some interval, say $(a,t)$. 

We need to determine the point $t$, after which $\Psi>\frac{1}{2}\Phi$. To minimize the error between $\Psi$ and $Q$, we need to match the derivatives of $\Psi$ and $Q$ at $t$. This leads to the condition $\frac{1}{2}(t-a)=(t-b),$ which gives $t=2b-a.$

For $x>t$, $\Psi>\frac{1}{2}\Phi$, thus $\Psi''=1$. To ensure $C^{1,1}$-regularity of $\Psi$, we define $\Psi=\frac{1}{2}(x-b)^2+\frac{1}{2}(b-a)^2.$

This gives the approximate solution on $\R$. See Figure \ref{OneDConstruction1}.
\end{eg}

\begin{figure}[ht]
\includegraphics[width=0.8\linewidth]{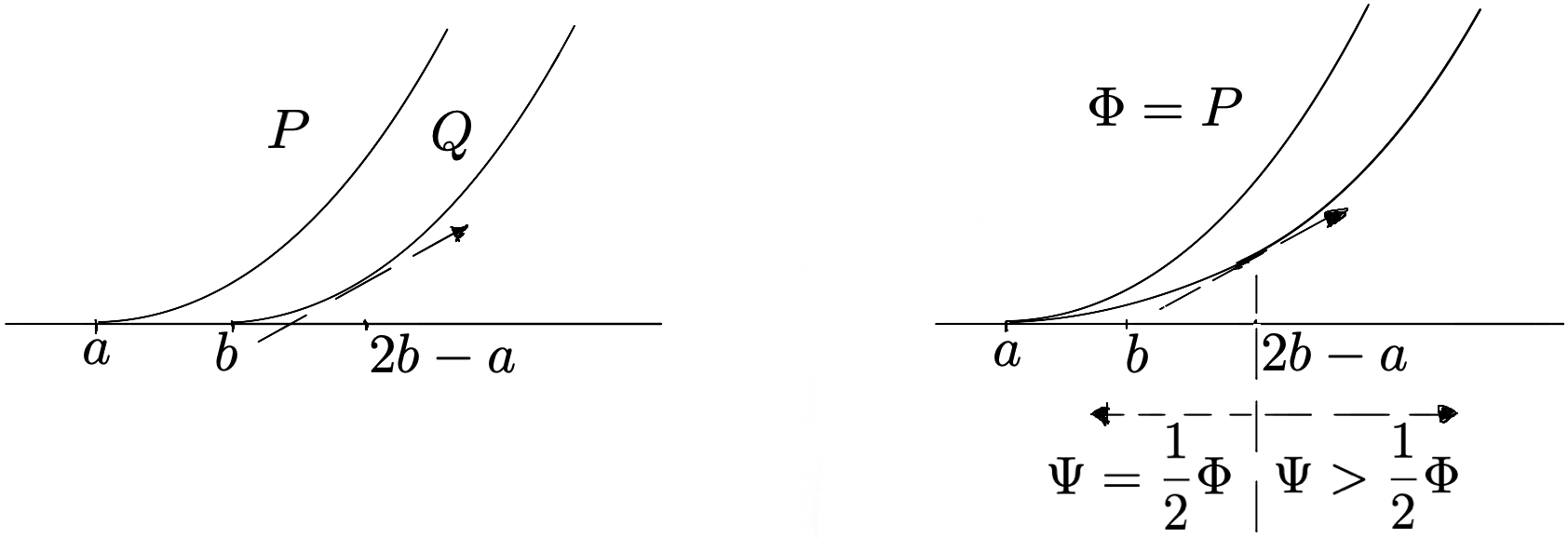}
\caption{Approximate solution on $\R$.}
\label{OneDConstruction1}
\end{figure}

In general dimensions,  the starting point is two half-space profiles $\Pabab$ and $\Qabab$ as in \eqref{Pabab1}. With \eqref{Symmetric}, it suffices to work in the $(x_1,x_2)$-plane. 

 The idea is to  follow  Example \ref{1DEx} for each fixed $x_2$. Such a line intersects  $\{x\cdot\alpha=a\}$ at $x_1=\frac{a-\alpha_2x_2}{\alpha_1}$ and intersects  $\{x\cdot\beta=b\}$ at $x_1=\frac{b-\beta_2x_2}{\beta_1}$. The smaller value between the two takes the place of $a$ as in Example \ref{1DEx}, while the larger one takes the role of $b$.  Then the free boundary point $2 b -a$ lies either on the line $\{(2\alpha-\beta)\cdot x=2a-b\}$ or $\{(2\beta-\alpha)\cdot x=2b-a\}.$ See Figure \ref{TwoDConstruction}.

\begin{definition}\label{ApproxSol1}
Corresponding to  $\PQabab$ as in \eqref{Pabab1}, the \textit{approximate solution} $(\Phi,\Psi)$ is defined as follows:
\begin{enumerate}
\item{Inside $\{\alpha_2x_2\ge \frac{a-b}{2}\}=\{x \cdot \alpha -a \ge x \cdot \beta-b\}$, $$\Phi=\frac{1}{2}\max\{x\cdot\alpha-a,0\}^2,$$ and $$\Psi=\begin{cases}\frac{1}{4}\max\{x\cdot\alpha-a,0\}^2 &\text{ for $(2\beta-\alpha)\cdot x<2b-a$,}\\ \frac{1}{2}(x\cdot\beta-b)^2+\frac{1}{2}(2\alpha_2x_2-a+b)^2 &\text{ for $(2\beta-\alpha)\cdot x\ge2b-a$.}
\end{cases}$$}
\item{Inside $\{\alpha_2x_2< \frac{a-b}{2}\}=\{x \cdot \alpha -a < x \cdot \beta-b\},$
$$\Phi=\begin{cases}\frac{1}{4}\max\{x\cdot\beta-b,0\}^2 &\text{ for $(2\alpha-\beta)\cdot x<2a-b,$}\\ \frac{1}{2}(x\cdot\alpha-a)^2+\frac{1}{2}(2\alpha_2x_2-a+b)^2 &\text{ for $(2\alpha-\beta)\cdot x\ge 2a-b$,}
\end{cases}$$ and 
$$\Psi=\frac{1}{2}\max\{x\cdot\beta-b,0\}^2.$$}
\end{enumerate}
\end{definition}

 \begin{figure}[ht]
\includegraphics[width=0.8\linewidth]{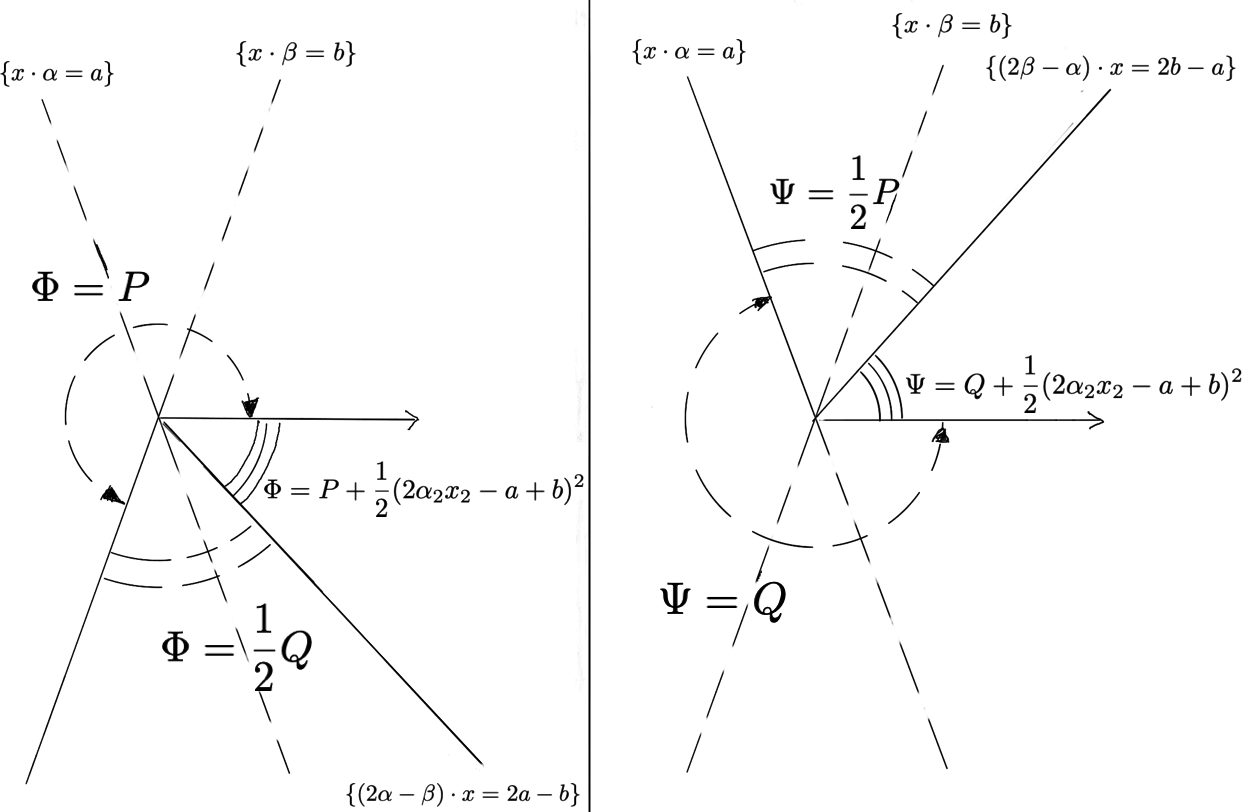}
\caption{Approximate solution in $\R^2$.}
\label{TwoDConstruction}
\end{figure}

The contact situation of the approximate solution is depicted Figure \ref{ContactApprox1}.
\begin{figure}[h]
\includegraphics[width=0.6\linewidth]{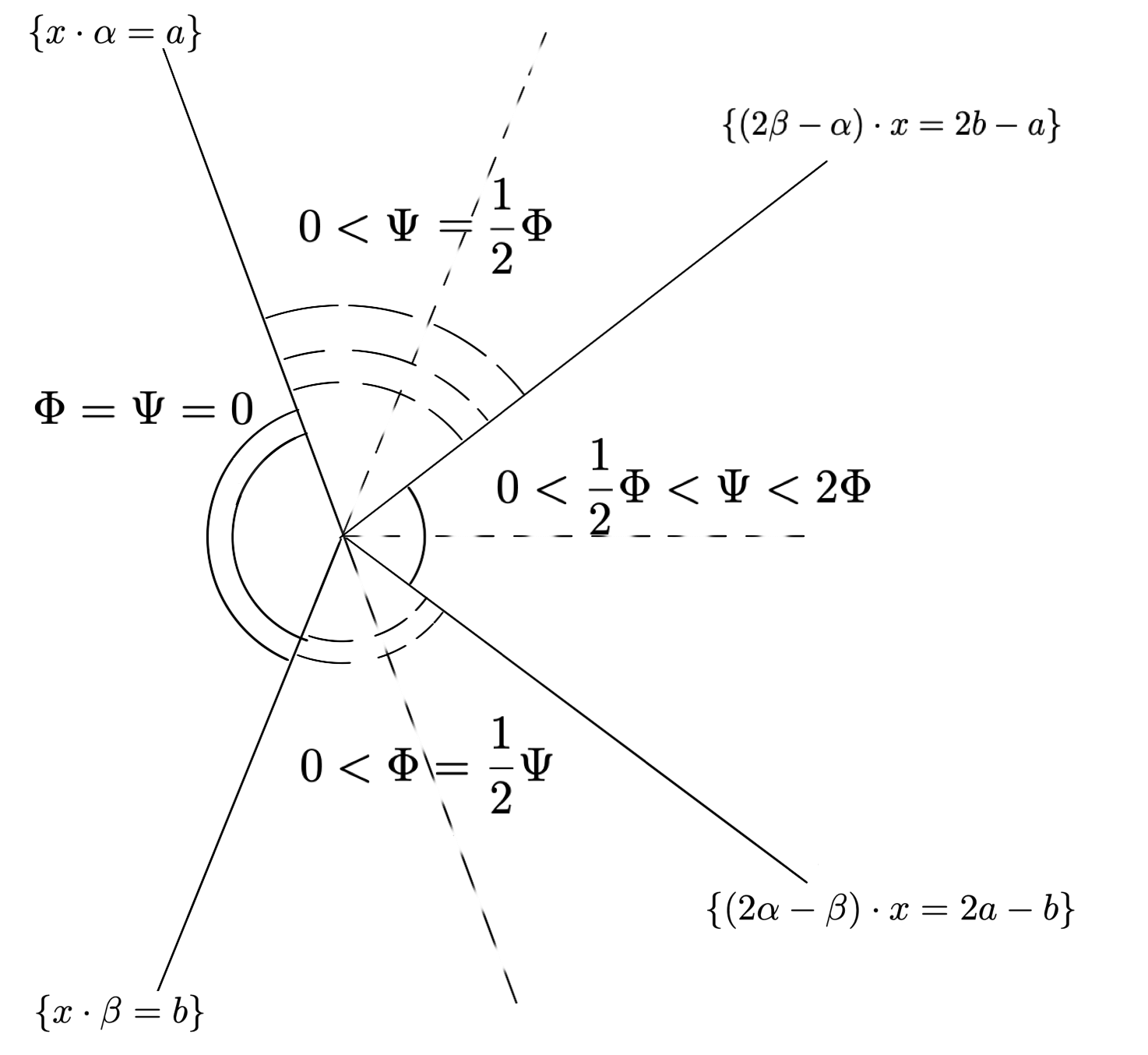}
\caption{Contacting situation of $\PP$.}
\label{ContactApprox1}
\end{figure}

These are approximate solutions in the sense described by the following lemma.  
\begin{lemma}\label{PropOfApproxSol1}
Let $(\Phi,\Psi)$ be the approximate solution defined above.  

Then $\Phi$ and $\Psi$ are $C^{1,1}$ functions. 

Moreover, there is a dimensional constant $C$ such that $$|\PP-(P,Q)|\le C(|\alpha-\beta|^2+|a-b|^2) \text{ in $B_1$,}$$and$$\PP\in\Subsol(\R^d) \text{ and }(1-C|\alpha-\beta|^2)\PP\in\Supsol(\R^d).$$
\end{lemma} 
The  classes $\Subsol$ and $\Supsol$ are defined in Definition \ref{ReSol}.

\begin{proof}
By definition, we have $\Phi$ and $\Psi$ are $C^{1,1}$. Moreover,  $\Delta\Phi\le 1+4\alpha_2^2$. Thus $(1-C|\alpha-\beta|^2)\Delta\Phi\le 1.$ 

On the other hand, in $\{\Phi>\frac{1}{2}\Psi\}$, we either have 
\begin{equation}\label{phiP}
\Phi=P \quad \mbox{ or} \quad  \Phi=P+\frac{1}{2}((\alpha-\beta)\cdot x-(a-b))^2.
\end{equation} 
In both cases, we have $\Delta\Phi\ge 1.$ 
Similar arguments apply to $\Psi$, and we have $$\PP\in\Subsol(\R^d)\text{ and }(1-C|\alpha-\beta|^2)\PP\in\Supsol(\R^d).$$

Now we estimate $|\Phi-P|$. In the set $$E:=\{(2\alpha-\beta)\cdot x\le 2a -b\} \cap \{x\cdot\beta>b\},$$ we have
$$0\le (x\cdot \alpha -a)^+\le \frac 12 (x\cdot\beta-b) \le x\cdot(\beta-\alpha)-(b-a),$$
thus $$|\Phi-P|=|\frac{1}{2}Q-P|\le C(|\alpha-\beta|^2+|b-a|^2).$$
In the complement $E^c$, \eqref{phiP} holds, and the inequality above remains valid in the whole $B_1$. A similar estimate holds for $|\Psi-Q|.$
\end{proof} 

These approximate solutions lead to fine estimate of solutions:
\begin{lemma}\label{TrappingByTranslations}
Suppose $(u,w)\in\mathcal{R}(\alpha,\beta;a,b;\eps)$ in $B_1$.

Let $\PP$ denote the approximate solution corresponding to $\PQabab$ as in \eqref{Pabab1}.
Then there are dimensional constants $A$ and $\eps_d$ such that
$$\PP(\cdot-A\eps\alpha)\le (u,w)\le\PP(\cdot+A\eps\alpha) \text{ in $B_{1/2}$}$$ if $\eps<\eps_d$.
\end{lemma} Recall the notation for comparison between pairs of functions as in \eqref{ComparisonBetweenPairs}.

We can replace $\alpha$ by $\beta$ and get the same comparisons. 

\begin{proof}We prove the upper bound. The lower bound follows from a similar argument.  The strategy is to apply Lemma \ref{FirstComparison} to translations of the approximate solution.

By Lemma \ref{Localization1}, we have $|a|,|b|,|\alpha-\beta|\le C\eps^{1/2}$. Thus Lemma \ref{PropOfApproxSol1} implies that $$(1-C\eps)\PP\in\Supsol(\R^d)$$ and $$|\PP-(P,Q)|\le C\eps \text{ in $B_1$.}$$

Define $$(F,G)=(1-C\eps)\PP(\cdot+A\eps\alpha),$$ where $A$ is a large constant to be chosen. 

Then $(F,G)\in\Supsol(B_1).$

For $\eps$ small, both $\Psi$ and $\Phi$ are non-decreasing in the direction along $\alpha$. Thus  we have the following comparison
\begin{align*}
(F,G)&\ge(1-C\eps)\PP\\&\ge(1-C\eps)((P,Q)-C\eps)\\&\ge(P,Q)-C\eps\\&\ge(u,w)-C\eps
\end{align*}inside $B_{3/4}$. 
Note that $$\{F>\frac{1}{4d}\}\subset\{\Phi(\cdot+A\eps\alpha)>\frac{1}{4d}\}\subset\{x\cdot\alpha-a>c\}$$ for a small dimensional constant $c>0.$ On the last set, we have $$\Phi(x+A\eps\alpha)-\Phi(x)\ge A\eps(x\cdot\alpha-a)\ge cA\eps.$$Thus on $\{F>\frac{1}{4d}\}\cap B_{3/4}$ we have
$$F=(1-C\eps)\Phi(x+A\eps\alpha)\ge \Phi(x)+cA\eps-C\eps\ge u+cA\eps-C\eps.$$ By choosing $A$ large, depending only on the dimension, we have $$F\ge u+10d\eps \text{ in $\{F>\frac{1}{4d}\}\cap B_{3/4}$.}$$ Similar comparison holds between $G$ and $w$ on $\{G>\frac{1}{4d}\}\cap B_{3/4}.$

Combining all these, we can apply Lemma \ref{FirstComparison} to get $(u,w)\le(F,G) \text{ on $B_{1/2}$.}$ This gives the desired upper bound. 
\end{proof} 

It is convenient to use the orthonormal basis for $\R^d$ containing $\alpha$. To simplify our notations, we introduce the following:
\begin{notation}\label{RotatedBasis}Let  $\{\alpha, \xi^2,\xi^3,\dots\xi^d\}$ be the orthonormal basis for $\R^d$ with $$\xi^2=-\alpha_2 e^1+\alpha_1 e^2,$$ and $\xi^k=e^k$ for $k\ge 3.$ 

 Let $x_k'$ denote the coordinate in the $\xi^k$ direction.  \end{notation} 

\subsection{Free boundary regularity when $\{x\cdot\alpha=a\}$ and $\{x\cdot\beta=b\}$ are well-separated} In this subsection we prove the $C^{1,\alpha}$-regularity of $\Gamma_u$ when the two hyperplanes are well-separated in $B_{2\rho_0}$. This is  alternative (1) in  Proposition \ref{IOF1}.

When the two hyperplanes are well-separated, the free boundaries $\Gamma_u$ and $\Gamma_w$ are at a definite distance to each other. Effectively, we are dealing with a single obstacle problem. Thus we can apply the result from Appendix A. 

Note that $\rho_0$ will be chosen in the next subsection, depending only on the dimension $d$. It suffices to prove the result at unit scale. 

Depending on the relative position of the hyperplanes, there are two cases to consider.  

 We first deal with the case when $b<a$. See Figure \ref{WellSeparated1}.

\begin{lemma}\label{FollowingTheReferee}
Suppose $(u,w)\in\mathcal{R}(\alpha,\beta;a,b;\eps)$ in $B_2$ for some $\eps<\eps_d.$

If  $a-b-4\alpha_2>M\eps^{3/4},$ then $$\Gamma_u\cap B_{3/2}\subset\{w>\frac{1}{2}u\}, \text{ and } \Gamma_w\cap B_{3/2}\subset\{u=\frac{1}{2}w\}.$$ Moreover, both free boundaries $\Gamma_u$ and $\Gamma_w$ are $C^{1,\alpha}$-hypersurfaces in $B_{1}$ with $C^{1,\alpha}$-norm  bounded by $C\eps$.

Here $\eps_d$, $M$, and $C$ are dimensional constants.
\end{lemma} 

The assumption \eqref{Symmetric} is still in effect. 
\begin{figure}[h]
\includegraphics[width=0.6\linewidth]{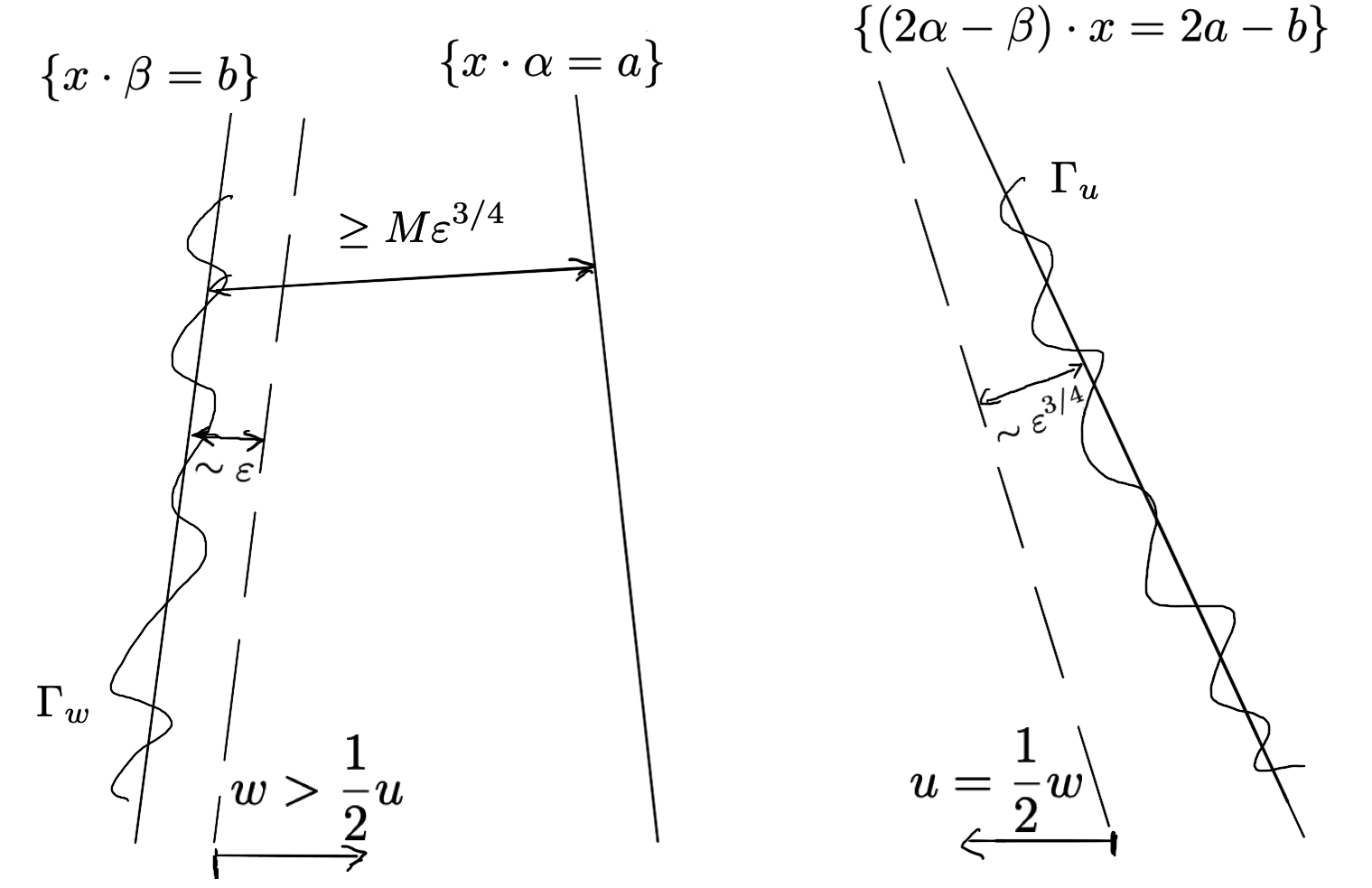}
\caption{Well-separated free boundaries.}
\label{WellSeparated1}
\end{figure}

\begin{proof}
Let $\PP$ be the approximate solution  corresponding to  $\PQabab$.

\textit{Step 1: Separation of free boundaries.}

Note that $a-b-4\alpha_2>M\eps^{3/4}$ implies $\alpha_2x_2\le\frac{a-b}{2}$ inside $B_2$. Thus we are always in the second alternative  in Definition \ref{ApproxSol1}. 

 In particular, $\Psi\ge\Phi\text{ in $\R^d.$}$

Using Lemma \ref{TrappingByTranslations}, inside $B_{3/2}$, we have
\begin{align*}
2w(x)-u(x)&\ge 2\Psi(x-A\eps\beta)-\Phi(x+A\eps\beta)\\&\ge2\Psi(x-A\eps\beta)-\Psi(x+A\eps\beta).\end{align*}
This implies \begin{equation}\label{Separation1}2w>u \text{ in $B_{3/2}\cap\{x\cdot\beta-b>3A\eps\}$}.\end{equation}

On the other hand, inside $\{(2\alpha-\beta)\cdot x<2a-b\}$, by Definition \ref{ApproxSol1}, $$2\Phi(x-A\eps\alpha)=\Psi(x-A\eps\alpha).$$ Thus  in $B_{4/3}\cap \{x\cdot\alpha<2a-b-2\alpha_2x_2\}$, Lemma \ref{TrappingByTranslations} gives $$2u(x)-w(x)\le 2\Phi(x+A\eps\alpha)-\Psi(x-A\eps\alpha)\le 2\Phi(x+A\eps\alpha)-2\Phi(x-A\eps\alpha)\le C\eps^{3/2}.$$ Lemma \ref{NonDegeneracy} implies \begin{equation}\label{Separation2}2u=w \text{ in $B_{3/2}\cap \{(2\alpha-\beta)\cdot x\le 2a-b-C\eps^{3/4}\}.$}\end{equation}

Under our assumption $a-b-4\alpha_2>M\eps^{3/4},$ we have
\begin{equation}\label{ChainOfInclusion}\{x\cdot\beta-b\le 3A\eps\}\subset\{x\cdot\alpha\le a\}\subset\{(2\alpha-\beta)\cdot x\le 2a-b-C\eps^{3/4}\} \text{ inside $B_{3/2}$}\end{equation} if $M$ is large. Thus \eqref{Separation1} and \eqref{Separation2} imply
$$\Gamma_u\cap B_{3/2}\subset\{w>\frac{1}{2}u\}, \text{ and } \Gamma_w\cap B_{3/2}\subset\{u=\frac{1}{2}w\}.$$

\textit{Step 2: Regularity of $\Gamma_u$.}

With \eqref{Separation1} and \eqref{Separation2}, we have $$\{u>\frac{1}{2}w\}\cap B_{3/2}\subset\{w>\frac{1}{2}u\}\cap B_{3/2}.$$

As a result, $$\Delta(2u-w)=\mathcal{X}_{\{2u-w>0\}}\text{ in $B_{3/2}.$}$$  
Moreover,  if we define $\alpha'=\frac{\alpha_1e^1+3\alpha_2e^2}{|\alpha_1e^1+3\alpha_2e^2|}$, then  $$|(2u-w)-\frac{1}{2}\max\{x\cdot\alpha'-2a+b,0\}^2|<C\eps \text{ in $B_{3/2}$.}$$ An application of Theorem \ref{ObReg} gives the desired regularity of $\Gamma_u$.

\textit{Step 3: Regularity of $\Gamma_w$.}

Since $2u=w$ in $B_{3/2}\cap\{(2\alpha-\beta)\cdot\le 2a-b-C\eps^{3/4}\}\supset B_{3/2}\cap\{x\cdot\alpha\le a\}$, the following is still a solution to the system of obstacle problems in $B_{3/2}$:
$$(\tilde{u},\tilde{w})=\begin{cases}
(u,w) &\text{ in $\{x\cdot\alpha\le a\}$,}\\
(\frac{1}{2}w,w) &\text{ in $\{x\cdot\alpha> a\}$.}
\end{cases}$$

Note that  \eqref{Separation1} and  \eqref{ChainOfInclusion} imply $\Gamma_w\cap B_{3/2}\subset\{x\cdot\alpha\le a\}$. Inside this set, $(\tilde{u},\tilde{w})=(u,w)$. Thus $\Gamma_w\cap B_{3/2}=\Gamma_{\tilde{w}}\cap B_{3/2}$. Consequently, it suffices to prove the desired regularity for $\Gamma_{\tilde{w}}\cap B_1.$

To this end, note that $\tilde{w}=2\tilde{u}$ inside $B_{3/2}$, we have  $$\Delta\tilde{w}=\mathcal{X}_{\{\tilde{w}>0\}} \text{ in $B_{3/2}$}$$ with $$\Gamma_{\tilde{w}}\cap B_{3/2}=\partial\{\tilde{w}>0\}\cap B_{3/2}.$$Therefore, Theorem  \ref{ObReg} implies $\Gamma_{\tilde{w}}\cap B_1$ is a $C^{1,\alpha}$-hypersurface with $C^{1,\alpha}$-norm  bounded by $C\eps.$\end{proof} 

 The second case is when $a<b$. To get the following, we just need to switch the roles of $u$ and $w$ and apply the previous lemma.
\begin{lemma}
Suppose $(u,w)\in\mathcal{R}(\alpha,\beta;a,b;\eps)$ in $B_2$ for some $\eps<\eps_d.$

If  $b-a-4\alpha_2>M\eps^{3/4},$ then $$\Gamma_u\cap B_{3/2}\subset \{w=\frac{1}{2}u\},$$ and $\Gamma_u$ is a $C^{1,\alpha}$-hypersurface in $B_{1}$ with $C^{1,\alpha}$-norm  bounded by $C\eps$.

Here $\eps_d$, $M$, and $C$ are dimensional constants.
\end{lemma} 

\subsection{Improvement of approximation and angle when  $\{x\cdot\alpha=a\}$ and $\{x\cdot\beta=b\}$ are close} In this subsection, we prove the second alternative in Proposition \ref{IOF1}.  For this alternative, we have 
\begin{equation}\label{Assum2}
|a-b|<\rho_0|\alpha-\beta|+M\rho_0\eps^{3/4},\end{equation} where $M$ is the constant from the previous subsection, and $\rho_0$ is a dimensional constant to be chosen.

Assumption \eqref{Symmetric} is still in effect. 

There are two results to prove. Firstly, we show an improvement of approximation at a small scale if $|\alpha-\beta|$ is less than $2\delta\eps^{1/2}.$ Secondly, we show that once $|\alpha-\beta|$ reaches the critical value $\delta\eps^{1/2}$, we can improve the angle by a definite amount at a smaller scale. 

\begin{lemma}\label{ImprovementOfApproximation1}
Suppose for parameters satisfying $|\alpha-\beta|<2\delta\eps^{1/2}$ and \eqref{Assum2}, we have $(u,w)\in \mathcal{R}(\alpha,\beta;a,b;\eps)$ in $B_1$ for some $\eps<\eps_d$.

Then $$(u,w)\in\mathcal{R}(\alpha',\beta';\frac{1}{2}\eps) \text{ in $B_{\rho_1}$}$$ with $|\alpha'-\alpha|+|\beta'-\beta|\le C\eps.$

Here $\eps_d$, $\delta$, $\rho_1$ and $C$ are dimensional constants. \end{lemma} 

\begin{proof}
Let $\PP$ denote the approximate solution corresponding to $\PQabab$ as in \eqref{Pabab1}.

Under the assumptions, we have $$|a-b|<2\delta\eps^{1/2}+C\eps^{3/4}.$$ In particular, Lemma \ref{PropOfApproxSol1} implies \begin{equation}\label{Smal}|\PP-(P,Q)|\le C\eps(\delta^2+\eps^{1/2}) \text{ in $B_1$.}\end{equation}

Note that $2u-w\ge 2P-Q-3\eps$ implies $$u>\frac{1}{2}w \text{ in $B_{1}\cap\{x_1'\ge a +C\eps^{1/2}\}$},$$where $x_1'$ is the coordinate function introduced in Notation \ref{RotatedBasis}. Consequently, if we define $$\hu=\frac{1}{\eps}(u-P),$$ then $$\Delta\hu=0 \text{ in $B_{1}\cap\{x_1'\ge a+ C\eps^{1/2}\}.$ }$$

Meanwhile, on $B_{3/4}\cap\{x_1' \le a+ C\eps^{1/2}\}$,  Lemma \ref{TrappingByTranslations} gives $$u(x)-\Phi(x)\le\Phi(x+A\eps\alpha)-\Phi(x)\le C\eps^{3/2},$$ and  $$u(x)-\Phi(x)\ge- C\eps^{3/2}.$$ Combining these with \eqref{Smal}, we have $$|\hu|\le C(\delta^2+\eps^{1/2}) \text{ in $B_{3/4}\cap\{x_1'\le a+ C\eps^{1/2}\}$.}$$

Let $h$ denote the solution to the following problem 
$$\begin{cases}
\Delta h=0 &\text{ in $B_{3/4}\cap\{x_1' > a+ C\eps^{1/2}\}$, }\\
h=\hu &\text{ along $\partial B_{3/4}\cap\{x_1'>  a+C\eps^{1/2}\}$, }\\
h=0 &\text{ along $B_{3/4}\cap\{x_1'=a+ C\eps^{1/2}\}$.}
\end{cases}$$ Then $$|\hu-h|\le C(\delta^2+\eps^{1/2}) \text{ in $B_{3/4}\cap\{x_1'> a+C\eps^{1/2}\}$.}$$

Boundary regularity of $h$ gives $$|h-\gamma_1(x_1'-a-C\eps^{1/2})-(x_1'-a-C\eps^{1/2})\sum_{k\ge 2}\gamma_kx_k'|\le Cr^3 \text{ in $B_r\cap\{x_1'>a+ C\eps^{1/2}\}$}$$ for some  bounded constants $\gamma_k$ and $r<1/2.$

If we define $$\alpha'=\frac{\alpha+\eps\sum_{k\ge 2}\gamma_k\xi^k}{|\alpha+\eps\sum_{k\ge 2}\gamma_k\xi^k|}$$ and $a'=a+\eps\gamma_1$, then $|\alpha-\alpha'|+|a'-a|\le C\eps$, and the previous estimate leads to
$$|u-\frac{1}{2}(x\cdot\alpha'-a')^2|\le C\eps(r^3+\delta^2+\eps^{1/2})$$
inside  $B_r\cap\{x_1'>a+ C\eps^{1/2}\}$ for $r<\frac{1}{2}.$ 

Here $\xi^k$ is the basis element in Notation \ref{RotatedBasis}.

Note that $u\le\Phi(\cdot+A\eps\alpha)\le C\eps^{3/2}$ in $B_{3/4}\cap\{x_1'\le a+ C\eps^{1/2}\}$, we have 
$$|u-P(\alpha')|\le C\eps(r^3+\delta^2+\eps^{1/2})$$ in $B_r$ if $r<\frac{1}{2}.$ Here we are using the notation in \eqref{Pabab1}.

From here, we choose $\rho_1$ small such that $C\rho_1^3<\frac{1}{9}\rho_1^2$, then $\delta$  and $\eps_d$ small such that $C(\delta^2+\eps_d^{1/2})<\frac{1}{9}\rho_1^2$, then 
$$|u-P(\alpha')|\le \frac{1}{3}\eps \rho_1^2 \text{ in $B_{\rho_1}$.}$$
A symmetric argument gives a similar estimate on $w$. 

Thus $(u,w)\in\mathcal{R}(\alpha',\beta';\frac{1}{2}\eps) \text{ in $B_{\rho_1}$.}$
\end{proof}

Now we give the improvement of the angle $|\alpha-\beta|$, once it reaches the critical level:
\begin{lemma}\label{ImprovementOfAngle1}
Suppose for parameters satisfying $\delta\eps^{1/2}<|\alpha-\beta|<2\delta\eps^{1/2}$ and \eqref{Assum2}, we have $(u,w)\in \mathcal{R}(\alpha,\beta;a,b;\eps)$ in $B_1$ for some $\eps<\eps_d$. 

Then $$(u,w)\in\mathcal{R}(\alpha',\beta';\frac{1}{32}\eps) \text{ in $B_{\rho_2}$}$$ with $|\alpha'-\alpha|+|\beta'-\beta|\le C\eps,$ and $$|\alpha'-\beta'|\le|\alpha-\beta|-\eps.$$

Here $\eps_d$, $\delta$, $\rho_2$, $C$  and $\rho_0$ (from \eqref{Assum2}) are dimensional. 
\end{lemma} 

\begin{proof}
Let $\PP$ denote the approximate solution corresponding to $\PQabab$ as in \eqref{Pabab1}.

Under the assumptions, we have $$\alpha_2\in(\frac{1}{2}\delta\eps^{1/2},\delta\eps^{1/2}).$$

We give the proof in two steps.  The first step covers the special case when $a=b=0.$ The second step deals with general $a$ and $b$ under assumption \eqref{Assum2}.

\textit{Step 1: The special case when  $a=b=0.$}

Define $\hu=\frac{1}{\eps}(u-P)$. 

As in the previous proof, we have $$\Delta\hu=0 \text{ in $B_{1}\cap\{x_1'\ge C\eps^{1/2}\}.$ }$$

Along $B_{3/4}\cap\{x_1'= C\eps^{1/2}\},$ we have $|u-\Phi|\le C\eps^{3/2}$ by Lemma \ref{TrappingByTranslations}. 

By definition of $\Phi$, $$\Phi-P=2(\alpha_2x_2)^2\mathcal{X}_{\{x_2<0\}} \text{ along $B_{3/4}\cap\{x_1'= C\eps^{1/2}\}$.}$$ Consequently, if we define $$\eta:=2(\alpha_2)^2/\eps\in(\frac{1}{2}\delta^2,2\delta^2),$$ then 
$$|\hu-\eta x_2^2\mathcal{X}_{\{x_2<0\}}|\le C\eps^{1/2} \text{ along $B_{3/4}\cap\{x_1'= C\eps^{1/2}\}$.}$$

With the coordinate system introduced in Notation \ref{RotatedBasis}, we define $h$ to  be the solution to the following problem 
\begin{equation}\label{ThatHarmonic}\begin{cases}
\Delta h=0 &\text{ in $B_{3/4}\cap\{x_1'> C\eps^{1/2}\}$,}\\
h=\hu &\text{ in $\partial B_{3/4}\cap\{x_1'> C\eps^{1/2}\}$,}\\
h=\eta(x_2')^2\mathcal{X}_{\{x_2'<0\}} &\text{ in $ B_{3/4}\cap\{x_1'=C\eps^{1/2}\}$.}
\end{cases}\end{equation}With $|x_2-x_2'|\le C\eps^{1/2}$ along $\{x_1'=C\eps^{1/2}\}$, the previous estimate gives \begin{equation}\label{Equation41}|\hu-h|\le C\eps^{1/2} \text{ in $B_{3/4}\cap\{x_1'> C\eps^{1/2}\}$.}\end{equation}

Let $\tilde{H}$ be the solution to $$\begin{cases}
\Delta \tilde{H}=0 &\text{ in $\R^d\cap\{x_1'> C\eps^{1/2}\}$,}\\
\lim_{|x|\to+\infty}\tilde{H}=0, &\\
\tilde{H}= (x_2')^2\mathcal{X}_{\{-1<x_2'<0\}} &\text{ in $ \R^d\cap\{x_1'=C\eps^{1/2}\}$.}
\end{cases}$$Then $(h-\eta \tilde{H})$ is a  bounded harmonic function in $B_{3/4}\cap\{x_1'> C\eps^{1/2}\}$ that vanishes along $\{x_1'= C\eps^{1/2}\}.$ Consequently, there are  bounded  constants $\gamma_k$ such that for $r<1/2$,
$$|(h-\eta \tilde{H})-(x_1'-C\eps^{1/2})(\gamma_1+\sum_{k\ge2}\gamma_kx_k')|\le Cr^3 \text{ in $B_r\cap\{x_1'>C\eps^{1/2}\}$.}$$

Comparing with the auxiliary function $H$ from Proposition \ref{AuxiliaryFunct}, we see that $\tilde{H}$ can be obtained from $H$ by a translation in $x_1'$-direction and a reflection in the $x_2'$-direction.  Therefore, Proposition \ref{AuxiliaryFunct} gives 
$$|\tilde{H}-A_1(x_1'-C\eps^{1/2})-A_2(x_1'-C\eps^{1/2})x_2'\log r|\le Cr^2 \text{ in $B_r\cap\{x_1'>C\eps^{1/2}\}$,}$$ for two dimensional constants $A_1,A_2>0.$

Note that we  flipped the sign in front of $A_2$ as the consequence of the reflection in the $x_2'$-direction.

Combining this with the previous estimate and \eqref{Equation41}, $$|u-\frac{1}{2}(x_1')^2-\eps(x_1')(\gamma_1+\sum_{k\ge2}\gamma_kx_k'+\eta A_2x_2'\log r)|\le C\eps(\delta^2 r^2+r^3+\eps^{1/2})$$
inside $B_r\cap\{x_1'>C\eps^{1/2}\}$. 

If we define $$\alpha'=\frac{\alpha+\eps\sum_{k\ge2}\gamma_k\xi^k+\eps\eta A_2\log r \xi^2}{|\alpha+\eps\sum_{k\ge2}\gamma_k\xi^k+\eps\eta A_2\log r \xi^2|}$$ and $a'=-\eps\gamma_1$, then 
$$|u-P(\alpha';a')|\le C\eps(\delta^2 r^2+r^3+\eps^{1/2})\text{ inside $B_r$},$$ where $P(\alpha';a')$ is defined in \eqref{Pabab1}.

To see the improvement of angle, we estimate $|\alpha'-e^1|$:
\begin{align*}
|\alpha'-e^1|\le& |\alpha+\eps\sum_{k\ge2}\gamma_k\xi^k+\eps\eta A_2\log r \xi^2-e^1|+C\eps^2\\\le& |\alpha+\eps\eta A_2\log r \xi^2-e^1|+\eps|\sum_{k\ge2}\gamma_k\xi^k|+C\eps^2\\=&|(\alpha_1-1-\alpha_2\eps\eta A_2\log r ,\alpha_2+\alpha_1\eps\eta A_2\log r )|\\&+\eps|\sum_{k\ge2}\gamma_k\xi^k|+C\eps^2, 
\end{align*}where we have used the definition of $\xi^2$ as in Notation \ref{RotatedBasis}. 

With $|\alpha-e^1|\le\eps^{1/2}$, this gives 
$$|\alpha'-e^1|\le |(\alpha_1-1 ,\alpha_2+\eps\eta A_2\log r )|+\eps|\sum_{k\ge2}\gamma_k\xi^k|+C\eps^{3/2}.$$

Since $\eta\in(\frac{1}{2}\delta^2,2\delta^2)$ while $|\sum_{k\ge2}\gamma_k\xi^k|$ is  bounded by a dimensional constant, we can find $\rho_2$ small, depending only on  the dimension, such that $$\eta A_2\log \rho_2<-|\sum_{k\ge2}\gamma_k\xi^k|-1.$$
Then
\begin{align*} |(\alpha_1-1 ,\alpha_2+\eps\eta A_2\log \rho_2 )|&\le |\alpha-e^1| +2\eps\eta A_2\log \rho_2+C\eps^{3/2}\\&\le -\eps-\eps|\sum_{k\ge 2}\gamma_k\xi^k|+C\eps^{3/2}.\end{align*}
Combining all these, we have $$|\alpha'-e^1|<|\alpha-e^1|-\eps$$ if $\eps_d$ is chosen small.

 If we fix $\delta$ small such that $C\eta<C\delta^2<\frac{1}{24}$, then choose $\rho_2$ small such that $C\rho_2^3<\frac{1}{24}\rho_2^2$, and finally choose $\eps_d$ small such that $C\eps_d^{1/2}<\frac{1}{24}\rho_2$,  then $$|u-P(\alpha';a')|<\frac{3}{24}\eps\rho_2^2 \text{ in $B_{\rho_2}$.}$$

Similarly, we can find $\beta'$ and $b'$ such that 
 $|w-Q(\beta';b')|<\frac{3}{24}\eps\rho_2^2 \text{ in $B_{\rho_2}$}$ and $$|\beta'-e^1|\le |\beta-e^1|-\eps.$$

Combining these, we have $$(u,w)\in\mathcal{R}(\alpha',\beta';\frac{1}{32}\eps) \text{ in $B_{2\rho_2}$ }$$with $$|\alpha'-\beta'|\le |\alpha-e^1|+|\beta-e^1|-2\eps\le|\alpha-\beta|-2\eps+C\eps^{3/2}<|\alpha-\beta|-\eps$$ if $\eps_d$ is small. 

This completes the proof for the case when $a=b=0.$

\textit{Step 2: General $a$ and $b$ satisfying \eqref{Assum2}.}

Under assumption \eqref{Assum2}, we have $|\frac{a-b}{2\alpha_2}|<\rho_0+M\rho_0\eps^{1/4}/\delta.$ Consequently, if $\eps_d$ is small, then $$\bar{x}:=\frac{a+b}{2\alpha_1}e^1+\frac{a-b}{2\alpha_2}e^2\in B_{2\rho_0}.$$ 

Note that $\bar{x}\cdot\alpha=a$ and $\bar{x}\cdot\beta=b$, by our assumptions on $(u,w)$,  we have $$|u-\frac{1}{2}\max\{(x-\bar{x})\cdot\alpha,0\}^2|<\eps \text{ and } |w-\frac{1}{2}\max\{(x-\bar{x})\cdot\beta,0\}^2|<\eps \text{ in $B_1$.}$$ Therefore, we can apply the result in the previous step to $(u,w)(\cdot-\bar{x})$. This gives $$(u,w)\in\mathcal{R}(\alpha',\beta'';\frac{1}{32}\eps) \text{ in $B_{2\rho_2}(\bar{x})$}$$with $|\alpha'-\beta'|<|\alpha-\beta|-2\eps.$

To conclude, simply note that if we choose $\rho_0<\frac{1}{4}\rho_2$, $B_{2\rho_2}(\bar{x})\supset B_{\rho_2}(0).$
\end{proof} 
This completes our proof for Proposition \ref{IOF1}.  In Section \ref{RegularPartSection}, it is used to prove the regularity of free boundaries near regular points as in Theorem \ref{MainResult2}.

\section{Improvement of flatness: Case 2}
In this section, we work with the system of obstacle problems introduced in Section 3. We give an improvement of flatness result relevant to singular points of type 1 in the $3$-membrane problem. 

According to Definition \ref{FreeBoundaryPoints}, around these points, the solution is approximated by unstable half-space solutions. We need to include translations and rotations of such profiles, that is, functions of the form
\begin{equation}\label{Pabab2}
\begin{cases}\PababTwo=&\frac{1}{2}\min\{x\cdot\alpha-a,0\}^2+\frac{1}{4}\max\{x\cdot\beta-b,0\}^2 \\
\QababTwo=&\frac{1}{4}\min\{x\cdot\alpha-a,0\}^2+\frac{1}{2}\max\{x\cdot\beta-b,0\}^2\end{cases}
\end{equation} for $\alpha,\beta\in\Sph$ and $a,b\in\R$.

We often write $P(\alpha,\beta)$ and $Q(\alpha,\beta)$ or even just $P$ and $Q$ for these profiles. 

In terms of the system of obstacle problems, we work with the following class of solutions:
\begin{definition}\label{SingSol}
For $\alpha,\beta\in\Sph$, $a,b\in\R$ and $\eps>0$, we say that $$(u,w)\in\Sabab \text{ in $B_r$}$$ if 
$$(u,w)\in\Sol(B_r) \text{ with } 0\in\Gamma_{u}\cap\Gamma_w,$$ and
$$|u-\PababTwo|<\eps r^2 \text{ and }|w-\QababTwo|<\eps r^2 \text{ in $B_r$}.$$
\end{definition} 
Recall that the class $\Sol$ is defined in Definition \ref{ReSol}.

We simply write $\Sab$ instead of $\Sabab$ if there is no need to emphasize $a$ and $b$.

Throughout this section, we still assume the symmetry assumption \eqref{Symmetric}.

Similar to Lemma \ref{Localization1}, the parameters are  bounded:
\begin{lemma}\label{Localization2}
Suppose $(u,w)\in\Sabab$ in $B_1$.  

Then $$|a|,|b|, |\alpha-\beta|<C\eps^{1/2}$$ for a dimensional constant $C$. 
\end{lemma} 

The main result in this section is:
\begin{proposition}[Improvement of flatness: Case 2]\label{IOF2}
There are  small positive constants $\delta, \eps_d$ and $\rho_k$  $(k=1,2),$ and  a large constant $C$, depending only on the dimension,  such that the following holds:

Suppose $$(u,w)\in\Sab \text{ in $B_1$}$$ with $$|\alpha-\beta|<2\delta\eps^{1/2}$$ for some $\eps<\eps_d.$ 

Then  there are $\alpha',\beta'\in\Sph$ with $|\alpha'-\alpha|+|\beta'-\beta|<C\eps$ such that $$(u,w)\in\mathcal{S}(\alpha',\beta';\eps/2) \text{ in $B_{\rho_1}$.}$$  

Moreover,  if $|\alpha-\beta|>\delta\eps^{1/2}$,   then there are $\alpha'',\beta''\in\Sph$ with $|\alpha''-\alpha|+|\beta''-\beta|<C\eps$ such that $$(u,w)\in\mathcal{S}(\alpha'',\beta'';\eps/2) \text{ in $B_{\rho_2}$}$$ and $$|\alpha''-\beta''|>|\alpha-\beta|+20\eps.$$  
\end{proposition} 

The most intriguing feature  is that the angle between the hyperplanes increases definitely once it reaches the critical level $\delta\eps^{1/2}$. This is a consequence of the instability of the unstable half-space solutions. Later, we need this instability to show that the angle never reaches the critical level at a singular point of type 1. 

We give the proof of Proposition \ref{IOF2} in the following subsections.  We omit proofs that are similar to the  ones in Section 4. 
\subsection{Approximate solutions}In this subsection, we build approximate solutions. We begin with the problem in one dimension.
\begin{eg}\label{1DEx2}
Suppose on $\R$, we are given two profiles 
$$P=\frac{1}{2}\min\{x-a,0\}^2+\frac{1}{4}\max\{x-b,0\}^2$$ and $$Q=\frac{1}{4}\min\{x-a,0\}^2+\frac{1}{2}\max\{x-b,0\}^2.$$ Our goal is to construct  an actual solution, $\PP$,   that best approximates $(P,Q).$

If $a\le b$, then $(P,Q)$ already solves the system of obstacle problem. In this case it suffices to take $\PP=(P,Q).$

If $a>b$, it is natural to take $\Phi=P$ and $\Psi=\frac{1}{2}\Phi$ for $x<<b.$ We need to determine the point $t$ such that $\Phi=2\Psi$ (thus $\Phi''=1$) for $x<t$,  and $\Phi=\frac{1}{2}\Psi$ (thus $\Phi''=\frac{1}{2}$ ) for $x>t$.  To approximate $P$, we need $\Phi'(t)=P'(t)$. This condition implies $(t-a)=\frac{1}{2}(t-b).$ We choose $t=2a-b$. Similar argument applies to $\Psi$.  

This gives the approximate solution $\PP$ on $\R$: $$\Phi(x)=\begin{cases}
\frac{1}{2}(x-a)^2+(a-b)^2 &\text{ if $x<2a-b$,}\\\frac{1}{4}(x-b)^2+\frac{1}{2}(a-b)^2&\text{ if $x\ge2a-b$,}
\end{cases}$$ and 
$$\Psi(x)=\begin{cases}
\frac{1}{4}(x-a)^2+\frac{1}{2}(a-b)^2 &\text{ if $x<2b-a$,}\\\frac{1}{2}(x-b)^2+(a-b)^2&\text{ if $x\ge2b-a$.}
\end{cases}$$This completes the construction in one dimension. See Figure \ref{OneDConstruction2}.\end{eg}

\begin{figure}[ht]
\includegraphics[width=0.6\linewidth]{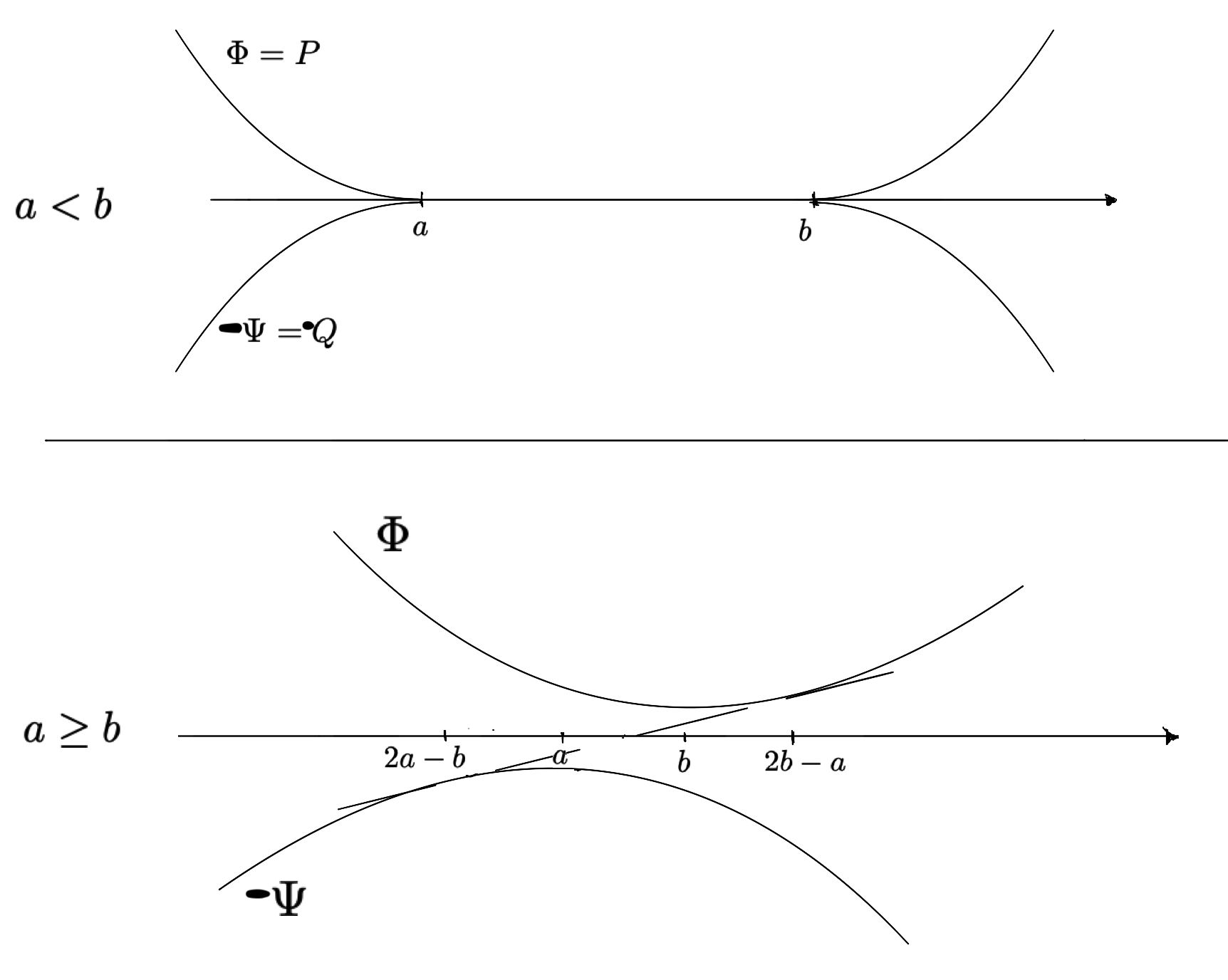}
\caption{Approximate solution on $\R$.}
\label{OneDConstruction2}
\end{figure}

For higher dimensions, we follow the same strategy along each hyperplane with fixed $x_2$.  See the strategy after Example \ref{1DEx}. 

Under our assumption \eqref{Symmetric}, this gives the following. See Figure \ref{TwoDConstruction2}.

 \begin{definition}\label{ApproxSol2}
Corresponding to $\PQabab$ as in \eqref{Pabab2},
the \textit{approximate solution} $$(\Phi,\Psi)=(\Phi,\Psi)(\alpha,\beta;a,b)$$ is defined as follows:
\begin{enumerate}
\item{Inside $\{\alpha_2x_2\ge \frac{a-b}{2}\}$, $$\Phi=P\text{ and }\Psi=Q;$$}
\item{Inside $\{\alpha_2x_2\le \frac{a-b}{2}\},$
$$\Phi=\begin{cases}\frac{1}{2}(x\cdot\alpha-a)^2+(2\alpha_2x_2-a+b)^2 &\text{ for $(2\alpha-\beta)\cdot x<2a-b,$}\\ \frac{1}{4}(x\cdot\beta-b)^2+\frac{1}{2}(2\alpha_2x_2-a+b)^2 &\text{ for $(2\alpha-\beta)\cdot x\ge 2a-b$,}
\end{cases}$$ and 
$$\Psi=\begin{cases}\frac{1}{4}(x\cdot\alpha-a)^2+\frac{1}{2}(2\alpha_2x_2-a+b)^2 &\text{ for $(2\beta-\alpha)\cdot x<2b-a,$}\\ \frac{1}{2}(x\cdot\beta-b)^2+(2\alpha_2x_2-a+b)^2 &\text{ for $(2\beta-\alpha)\cdot x>2b-a$.}
\end{cases}$$}
\end{enumerate}
\end{definition}

\begin{figure}[h]
\includegraphics[width=0.8\linewidth]{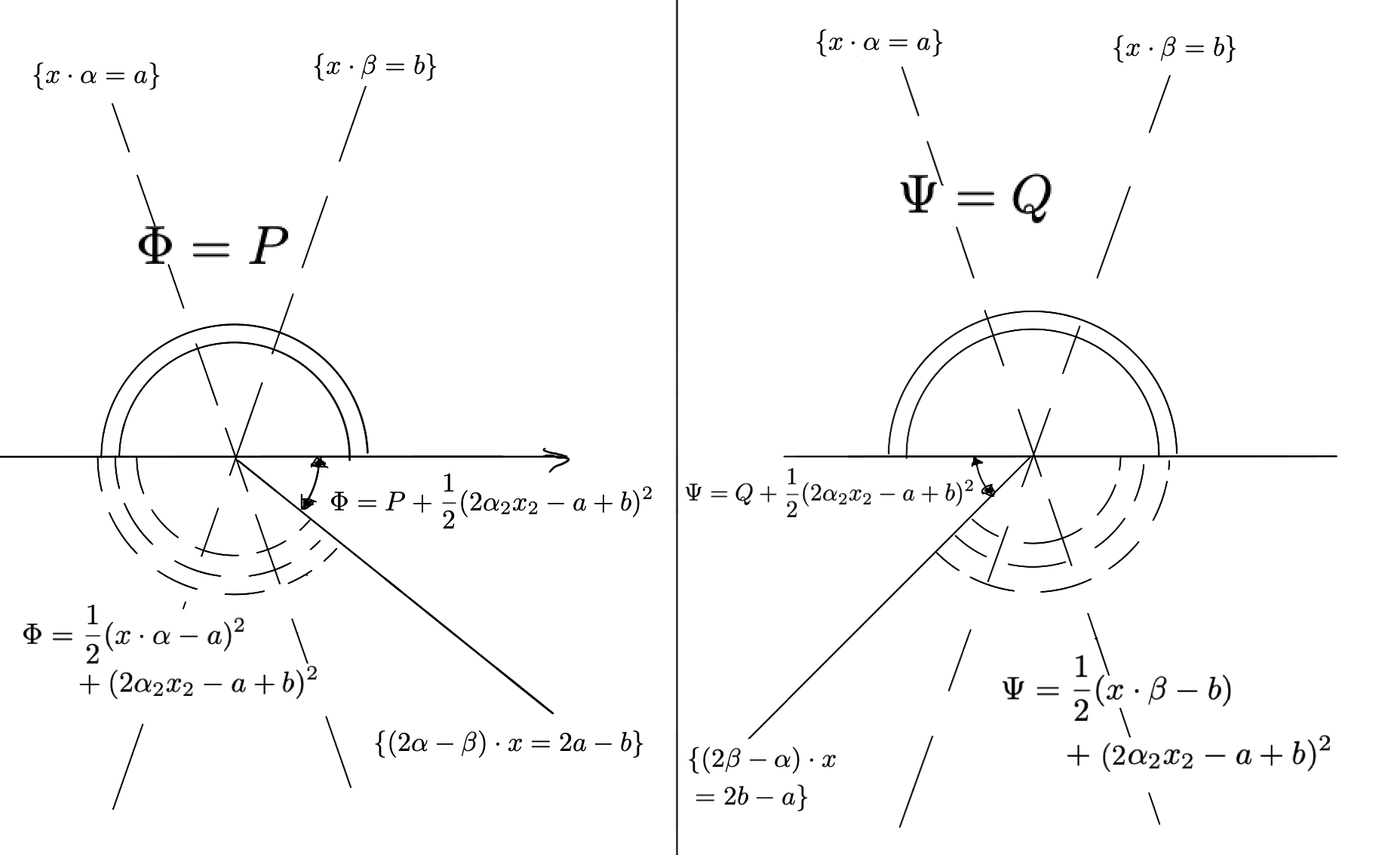}
\caption{Approximate solution in $\R^2$.}
\label{TwoDConstruction2}
\end{figure}

The contact situation of $\PP$ is depicted in Figure \ref{ContactApprox2}.
\begin{figure}[h]
\includegraphics[width=0.6\linewidth]{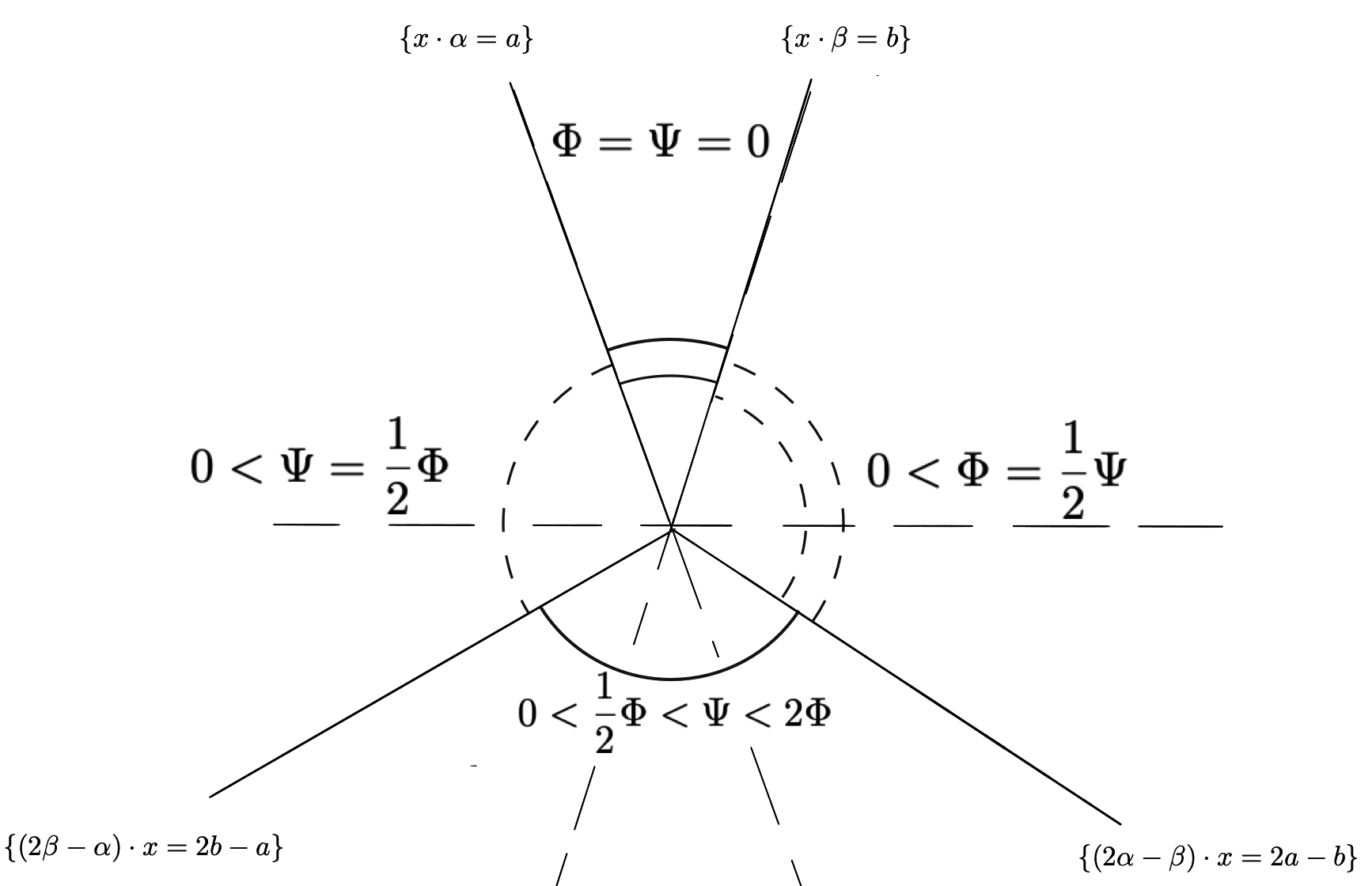}
\caption{Contacting situation of $\PP$.}
\label{ContactApprox2}
\end{figure}

Similar to Lemma \ref{PropOfApproxSol1}, we have
\begin{lemma}\label{PropOfApproxSol2}
Let $(\Phi,\Psi)$ be the approximate solution defined above.  

Then there is a dimensional constant $C$ such that $$|\PP-(P,Q)|\le C(|\alpha-\beta|^2+|a-b|^2) \text{ in $B_1$,}$$and$$\PP\in\Subsol(\R^d) \text{ and }(1-C|\alpha-\beta|^2)\PP\in\Supsol(\R^d).$$
\end{lemma} 

Similar to Lemma \ref{TrappingByTranslations}, using the notation as in Definition \ref{ApproxSol2}, we have
\begin{lemma}\label{TrappingByTranslations2}
Suppose $(u,w)\in\Sabab$ in $B_1$.

Then there are dimensional constants $A$ and $\eps_d$ such that
$$\begin{cases}&\Phi(\alpha,\beta;a-A\eps,b+A\eps)\le u\le\Phi(\alpha,\beta;a+A\eps,b-A\eps)\\
&\Psi(\alpha,\beta;a-A\eps,b+A\eps)\le w\le\Psi(\alpha,\beta;a+A\eps,b-A\eps)\end{cases}\text{ in $B_{1/2}$}$$ if $\eps<\eps_d.$
\end{lemma}

When there is no ambiguity, we simplify our notations by writing \begin{equation}\label{Barriers2}\begin{cases}&(\Phi^-,\Psi^-)=\PP(\alpha,\beta;a-A\eps,b+A\eps),\\  &(\Phi^+,\Psi^+)=\PP(\alpha,\beta;a+A\eps,b-A\eps).\end{cases}\end{equation}
\subsection{Improvement of approximation and angle}This subsection contains the proof of Proposition \ref{IOF2}. We divide the proposition into two statements. The first is an improvement of approximation result, assuming the angle $|\alpha-\beta|$ is small. The second is to show that this angle increases by a definite amount once it reaches the critical level.

We first give a finer bound on $a$ and $b$. This refinement is a consequence of our assumption that $0\in\Gamma_u\cap\Gamma_w.$ See Definition \ref{SingSol}.
\begin{lemma}\label{RefinedLocalization}
Suppose $(u,w)\in\Sabab$ in $B_1$.

 Then $$|a|+|b|\le C\eps$$ for a dimensional constant $C$.
\end{lemma} 

\begin{proof}
We use the notation in \eqref{Barriers2}.

Suppose $(a-A\eps)-(b+A\eps)>0$, then either $\Phi^-(0)=\frac{1}{2}(-a+A\eps)^2+(-a+b+2A\eps)^2$ or $\Phi^-(0)=\frac{1}{4}(-b-A\eps)+\frac{1}{2}(-a+b+2A\eps)^2$. 

In both cases, $0=u(0)\ge\Phi^-(0)$  gives the desired estimate.

Consequently, it suffices to consider the case when $a-b\le 2A\eps.$ 

In this case, the comparison $0=u(0)\ge\Phi^-(0)$ implies \begin{equation}\label{ABound}a\le A\eps \text{ and } b\ge -A\eps.\end{equation}

Suppose, on the contrary, that $a\le -M\eps$ for a large $M$. 

The previous estimate on $b$ implies $a-b+2A\eps\le-\frac{1}{2}M\eps$ if $M\gg A.$ Thus $$B_{\frac{1}{4}M\eps}\subset \{\alpha_2x_2\ge \frac{1}{2}(a-b+2A\eps)\}.$$Thus for both $\Phi^+$ and $\Psi^-$, only alternative (1) in Definition \ref{ApproxSol2} is relevant in $B_{\frac{1}{4}M\eps}.$ 

As a result, we have
\begin{equation*}
2u-w\le 2\Phi^+-\Psi^-\le  \frac{1}{4}MA\eps^2 \text{ in $B_{\frac{1}{4}M\eps}$.}\end{equation*}
With Lemma \ref{NonDegeneracy}, this implies $u=\frac{1}{2}w$ in a neighborhood of $0$ if $M$ is large, contradicting $0\in\Gamma_u.$

Consequently, $a\ge -C\eps$. Similar arguments give $b\ge -C\eps$. 

Combined with \eqref{ABound}, we have the desired estimate. \end{proof}

We now prove the improvement of approximation in Proposition \ref{IOF2}.
\begin{lemma}
Suppose $(u,w)\in\Sab$ in $B_1$ with $|\alpha-\beta|<2\delta\eps^{1/2}$ for some $\eps<\eps_d$.

Then for some $\alpha',\beta'\in\Sph$ satisfying $$|\alpha'-\alpha|+|\beta'-\beta|\le C\eps$$ we have $$(u,w)\in\mathcal{S}(\alpha',\beta';\frac{1}{2}\eps) \text{ in $B_{\rho_1}.$}$$Here $\delta, \eps_d, \rho_1$ and $C$ are dimensional constants. 
\end{lemma} 

The proof is similar to the proof of Lemma \ref{ImprovementOfApproximation1}. We omit some details. 

\begin{proof}
Let $(\Phi^{\pm},\Psi^{\pm})$ be the barriers as in \eqref{Barriers2}, corresponding to $\PQabab$ as in \eqref{Pabab2}.

Using $|u-P|\le\eps,|w-Q|\le\eps$ and  Lemma \ref{NonDegeneracy}, we have 
\begin{equation}\label{TouchingLeft}u>\frac{1}{2}w \text{ and } w=\frac{1}{2}u \text{ in $B_{7/8}\cap\{x_1'<-C\eps^{1/2}\},$}\end{equation}where $x_1'$ is the coordinate function in Notation \ref{RotatedBasis}.

Define $\hu=\frac{1}{\eps}(u-P)$, then $$\Delta\hu=0 \text{ in $B_{7/8}\cap\{x_1'<-C\eps^{1/2}\}.$}$$
Meanwhile,   Lemma \ref{TrappingByTranslations2} gives $$|\hu|\le C(\delta^2+\eps^{1/2}) \text{ along $B_{7/8}\cap\{x_1'=-C\eps^{1/2}\}.$}$$ 

Let $h$ be the solution to 
$$\begin{cases}
\Delta h=0 &\text{ in $B_{7/8}\cap\{x_1'<-C\eps^{1/2}\}$,}\\
h=\hu &\text{ in $\partial B_{7/8}\cap\{x_1'<-C\eps^{1/2}\}$,}\\
h=0 &\text{ in $B_{7/8}\cap\{x_1'=-C\eps^{1/2}\}$.}
\end{cases}$$Then we can use boundary regularity of $h$ to get 
$$|u-\frac{1}{2}(x_1'-a)^2-\eps(x_1'-a)(\gamma_1+\sum_{k\ge 2}\gamma_kx_k')|\le C\eps(r^3+\delta^2+\eps^{1/2})$$ in $B_{r}\cap\{x_1'<-C\eps^{1/2}\}$ for $r<1/2.$  Here $x_k'$ is the coordinate function in Notation \ref{RotatedBasis}.

If we define $$\alpha'=\frac{\alpha+\eps\sum_{k\ge2}\gamma_k\xi^k}{|\alpha+\eps\sum_{k\ge2}\gamma_k\xi^k|},$$ where $\xi^k$ is defined as in Notation \ref{RotatedBasis}, and $a'=a+\eps\gamma_1$, then $|\alpha'-\alpha|\le C\eps$ and $$|u-\frac{1}{2}(x\cdot\alpha'-a')^2|\le C\eps(r^3+\delta^2+\eps^{1/2})$$ in $B_{r}\cap\{x_1'<-C\eps^{1/2}\}$ for $r<1/2.$ 

Now with \eqref{TouchingLeft}, we have $$|w-\frac{1}{4}(x\cdot\alpha'-a')^2|\le C\eps(r^3+\delta^2+\eps^{1/2})$$ in $B_{r}\cap\{x_1'<-C\eps^{1/2}\}$ for $r<1/2.$ 

Similarly, we can find $\beta'$ and $b'$ with $|\beta'-\beta|\le C\eps$ such that 
$$|u-\frac{1}{4}(x\cdot\beta'-b')^2|\le C\eps(r^3+\delta^2+\eps^{1/2}),$$ and 
$$|w-\frac{1}{2}(x\cdot\beta'-b')^2|\le C\eps(r^3+\delta^2+\eps^{1/2})$$ in $B_{r}\cap\{x_1'>C\eps^{1/2}\}$ for $r<1/2.$ 

Combining these,  we have $$|u-P(\alpha',\beta')|+|w-Q(\alpha',\beta')|\le C\eps(r^3+\delta^2+\eps^{1/2}) $$ in $B_r$ for $r<1/2$. Recall that $P(\alpha',\beta')$ and $Q(\alpha',\beta')$ are defined  in \eqref{Pabab2}.

To conclude, we choose $\rho_1$ small such that $C\rho_1^3<\frac{1}{6}\rho_1^2$, $\delta$ small such that $C\delta^2<\frac{1}{6}\rho_1^2$, then $\eps_d$ small such that $C\eps_d^{1/2}<\frac{1}{2}\rho_1.$ Then the estimate above implies $$(u,w)\in\mathcal{S}(\alpha',\beta';\frac{1}{2}\eps) \text{ in $B_{\rho_1}$.}$$\end{proof} 

The next result is on the increase of angle:
\begin{lemma}\label{ImprovementOfAngle2}

Suppose $(u,w)\in\Sab$ in $B_1$ for some $\eps<\eps_d$ with $|\alpha-\beta|\in (\delta\eps^{1/2},2\delta\eps^{1/2}).$   

Then  there are $\alpha',\beta'\in\Sph$ with $|\alpha'-\alpha|+|\beta'-\beta|<C\eps$ such that $$(u,w)\in\mathcal{S}(\alpha',\beta';\eps/2) \text{ in $B_{\rho_2}$}$$ and $$|\alpha'-\beta'|>|\alpha-\beta|+20\eps.$$ 

Here $\eps_d$, $\delta$, $\rho_2$ and $C$ are dimensional. 
\end{lemma} 

The strategy is similar to the one in the proof of Lemma \ref{ImprovementOfAngle1}. We omit certain details. 

The key observation is that the difference $\PP-(P,Q)$ is a reflection of  what it was in the proof of Lemma \ref{ImprovementOfAngle1}. This drives the increase of the angle.

\begin{proof}
Let $\PQabab$ be as in \eqref{Pabab2}.
Let $\PP$ be the approximate solution  in Definition \ref{ApproxSol2}. Let $(\Phi^{\pm},\Psi^{\pm})$ be as in \eqref{Barriers2}. 

We again use the coordinate system introduced in Notation \ref{RotatedBasis}.

Define $\hu=\frac{1}{\eps}(u-P),$  then  $$\Delta\hu=0 \text{ in $B_{7/8}\cap\{x_1'<-C\eps^{1/2}\}.$}$$With Lemma \ref{TrappingByTranslations2}, we have $$\hu=\frac{1}{\eps}(\Phi-P)+O(\eps^{1/2})  \text{ along $B_{7/8}\cap\{x_1'=-C\eps^{1/2}\}.$}$$Now with $|a|,|b|\le C\eps$ from Lemma \ref{RefinedLocalization}, we have $$(\Phi-P)=4\alpha_2^2x_2^2\mathcal{X}_{\{x_2<0\}}+O(\eps^{3/2})\text{ along $B_{7/8}\cap\{x_1'=-C\eps^{1/2}\}.$}$$

Let $h$  be the solution to 
$$\begin{cases}
\Delta h=0 &\text{ in $B_{3/4}\cap\{x_1'<-C\eps^{1/2}\}$,}\\
h=\hu &\text{ in $\partial B_{3/4}\cap\{x_1'<-C\eps^{1/2}\}$,}\\
h=4\frac{\alpha_2^2}{\eps}(x_2')^2\mathcal{X}_{\{x_2'<0\}} &\text{ in $B_{3/4}\cap\{x_1'=-C\eps^{1/2}\}$.}
\end{cases}$$Then $$|\hu-h|\le C\eps^{1/2} \text{ in  $B_{3/4}\cap\{x_1'<-C\eps^{1/2}\}$.}$$

Note that $h$ is the reflection along $\{x_1'=0\}$ of the harmonic function in \eqref{ThatHarmonic}, we can use the same argument to get  
$$|u-\frac{1}{2}(x_1')^2-\eps(x_1')(\gamma_1+\sum_{k\ge2}\gamma_kx_k'-\frac{4\alpha_2^2}{\eps} A_2x_2'\log r)|\le C\eps(\delta^2 r^2+r^3+\eps^{1/2})$$
inside $B_r\cap\{x_1'<-C\eps^{1/2}\}$ for $r<1/2,$ where $A_2$ is the positive dimensional constant in Proposition \ref{AuxiliaryFunct}.

Note that the sign in front of $A_2$ is flipped due to the reflection. 

If we define $$\alpha'=\frac{\alpha+\eps\sum_{k\ge2}\gamma_k\xi^k-4\alpha_2^2A_2\log r \xi^2}{|\alpha+\eps\sum_{k\ge2}\gamma_k\xi^k-4\alpha_2^2A_2\log r \xi^2|}$$ and $a'=a+\eps\gamma_1$, then this implies 
$$|u-\frac{1}{2}(x\cdot\alpha'-a')^2|\le C\eps(\delta^2 r^2+r^3+\eps^{1/2})$$ inside $B_r\cap\{x_1'<-C\eps^{1/2}\}$.

Similar strategy applied to the region $\{x_1'>C\eps^{1/2}\}$ gives $$|u-\frac{1}{4}(x\cdot\beta'-b')^2|\le C\eps(\delta^2 r^2+r^3+\eps^{1/2})$$ inside $B_r\cap\{x\cdot\beta>C\eps^{1/2}\}$, where $\beta'$ is of the form 
$$\beta'=\frac{\beta+\eps\sum_{k\ge2}\gamma'_k\xi^k+4\alpha_2^2A_2\log r \xi^2}{|\beta+\eps\sum_{k\ge2}\gamma_k'\xi^k+4\alpha_2^2A_2\log r \xi^2|}.$$Using similar estimates as in the proof of Lemma \ref{ImprovementOfApproximation1}, we have 
$$\alpha'-\beta'=(2\alpha_2-8\alpha_2^2A_2\log r)e^2+\eps(\sum_{k\ge 3}(\gamma_k+\gamma_k')e^k)+O(\eps^{3/2}).$$

By choosing $\rho_2$ small, depending only on the dimensional constant $\delta$, we can ensure $$\delta^2A_2\log \rho_2\le-( |\sum_{k\ge 3}(\gamma_k+\gamma_k')e^k|+20).$$ As a result, $$|\alpha'-\beta'|\ge |2\alpha_2+2\eps|=|\alpha-\beta|+20\eps$$ if $\eps$ is small. 

This is the desired increase in angle. 

To get $(u,w)\in\mathcal{S}(\alpha',\beta';\frac{1}{2}\eps)$ in $B_{\rho_2}$, we proceed exactly like in the proof of Lemma \ref{ImprovementOfAngle1}.\end{proof}

This completes the proof of Proposition \ref{IOF2}.  In Section \ref{SingularPartSection}, this proposition is used to prove the regularity of $\SOne$ as in Theorem \ref{MainResult3}.

\section{Free boundary regularity near $\Reg$}\label{RegularPartSection}
Starting from this section, we return to the $3$-membrane problem in Definition \ref{Solution}, and give our proofs for the three theorems in Introduction. 


Theorem \ref{MainResult2} is a direct consequence of the following point-wise localization of the free boundary:
\begin{lemma}\label{LongIteration1}
Suppose that $\Trip$ is a solution to the $3$-membrane problem in $B_1$ with $0\in\Gamma_1$.

If we have, for some $e\in\Sph$ and  $\eps<\eps_d$, $$|u_1-\frac{1}{2}\max\{x\cdot e-a,0\}^2|\le \eps \text{ and }|u_3+\frac{1}{2}\max\{x\cdot e-a,0\}^2|\le \eps \text{ in $B_1$,}$$ then, up to a rotation, $$\Gamma_1\cap B_r\subset\{|x_1|<Cr(-\log r)^{-1}\}$$ for all $r<1/2.$

Here $\eps_d$ and $C$ are dimensional constants. 
\end{lemma} 

\begin{proof}
Define $u=u_1$ and $w=-u_3$. 

Due to the equivalence of the $3$-membrane problem and the system of obstacle problems, as in Remark \ref{EquivalenceBetweenProblems}, it suffices to prove that, up to a rotation, 
\begin{equation}\label{62}\Gamma_u\cap B_r\subset\{|x_1|<C r(-\log r)^{-1}\}\end{equation} for all $r<1/2.$ 

We prove this with an iterative scheme.

\textit{Step 1: Description of the iteration scheme.} 

As the starting point, we define $$u_0=u, w_0=w,$$ $$\alpha_0=\beta_0=e, a_0=b_0=a,$$ $$\eps_0=\eps \text{ and }r_0=1.$$ Then we have $$(u_0,w_0)\in\mathcal{R}(\alpha_0,\beta_0;a_0,b_0;\eps_0) \text{ in $B_1$.}$$ The class $\mathcal{R}$ is defined in Definition \ref{RegPointSol}. 

Suppose that we have completed the $k$th step in this iteration, that is, we have found 
$$(u_k,w_k)\in\mathcal{R}(\alpha_k,\beta_k;a_k,b_k;\eps_k) \text{ in $B_1$}$$ for some $\eps_k<\eps_d$ and $|\alpha_k-\beta_k|<2\delta\eps_k^{1/2}$, where $\eps_d$ and $\delta$ are the constants in Proposition \ref{IOF1}, then we proceed to the $(k+1)$th step as follows.

We consider three cases. 

\textit{The first case} is when $|a_0-b_0|>\rho_0|\alpha_k-\beta_k|+M\rho_0\eps_k^{3/4}$, that is, when the parameters fall into alternative 1 as in Proposition \ref{IOF1}. In this case, we terminate the iteration scheme. 

\textit{The second case} is when $|a_0-b_0|\le\rho_0|\alpha_k-\beta_k|+M\rho_0\eps_k^{3/4}$ and $|\alpha_k-\beta_k|<\delta\eps_k^{1/2}.$ In this case, we apply Lemma \ref{ImprovementOfApproximation1} to get $$(u_k,w_k)\in\mathcal{R}(\alpha',\beta';a',b';\frac{1}{2}\eps_k) \text{ in $B_{\rho_1}$}$$ for some $|\alpha'-\alpha_k|+|\beta'-\beta_k|+|a'-a_k|+|b'-b_k|\le C\eps_k$.  

Define $$u_{k+1}(x)=\frac{1}{\rho_1^2}u_k(\rho_1 x), w_{k+1}(x)=\frac{1}{\rho_1^2}w_k(\rho_1 x),$$
$$\alpha_{k+1}=\alpha', \beta_{k+1}=\beta', a_{k+1}=a', b_{k+1}=b', $$
$$\eps_{k+1}=\frac{1}{2}\eps_k \text{ and } r_{k+1}=\rho_1r_k.$$ Then 
$$(u_{k+1},w_{k+1})\in\mathcal{R}(\alpha_{k+1},\beta_{k+1};a_{k+1},b_{k+1};\eps_{k+1}) \text{ in $B_1$.}$$Note that in this case, we have \begin{align*}|\alpha_{k+1}-\beta_{k+1}|&\le|\alpha_{k}-\beta_k|+C\eps_k\\&<\delta\eps_k^{1/2}+C\eps_k\\&\le\sqrt{2}\delta\eps_{k+1}^{1/2}+C\eps_{k+1}.\end{align*} This implies $|\alpha_{k+1}-\beta_{k+1}|<2\delta\eps_{k+1}^{1/2}$ if $\eps_d$ is small. 

This completes the $(k+1)$th step in the second case.

\textit{The third case} is when $|a_0-b_0|\le\rho_0|\alpha_k-\beta_k|+M\rho_0\eps_k^{3/4}$ and $|\alpha_k-\beta_k|\ge\delta\eps_k^{1/2}.$ In this case, define $$\tilde{\eps}=(\frac{|\alpha_k-\beta_k|}{2\delta})^2.$$ Then \begin{equation}\label{61}\frac{1}{4}\eps_{k}\le \tilde{\eps}\le\eps_k.\end{equation} Consequently,  we have $$(u_k,w_k)\in\mathcal{R}(\alpha_k,\beta_k;a_k,b_k;4\tilde{\eps}) \text{ in $B_1.$}$$ 

By definition of $\tilde{\eps}$, we have $|\alpha_k-\beta_k|=\delta(4\tilde{\eps})^{1/2},$ thus  Lemma \ref{ImprovementOfAngle1} gives $$(u_k,w_k)\in\mathcal{R}(\alpha',\beta';a',b';\frac{1}{8}\tilde{\eps}) \text{ in $B_{\rho_2}$}$$ for some $|\alpha'-\alpha_k|+|\beta'-\beta_k|+|a'-a_k|+|b'-b_k|\le C\tilde{\eps}$ and $$|\alpha'-\beta'|\le|\alpha_k-\beta_k|-4\tilde{\eps}.$$

Define $$\eps_{k+1}=(\frac{|\alpha'-\beta'|}{2\delta})^2,$$then $$\eps_{k+1}\le (\frac{|\alpha_k-\beta_k|-4\tilde{\eps}}{2\delta})^2=(\tilde{\eps}^{1/2}-\frac{2}{\delta}\tilde{\eps})^2\le \eps_{k}-C\eps_{k}^{3/2}.$$ For the last comparison, we used \eqref{61}. 

A similar comparison, using $|\alpha'-\alpha_k|+|\beta'-\beta_k|\le C\tilde{\eps}$, implies \begin{equation}\label{remark2}\eps_{k+1}\ge \eps_k-C\eps_k^{3/2}. \end{equation}In particular, $\eps_{k+1}\ge\frac{1}{8}\tilde{\eps}$ if $\eps_d$ is small. Thus $$(u_k,w_k)\in\mathcal{R}(\alpha',\beta';a',b';\eps_{k+1}) \text{ in $B_{\rho_2}$}.$$ The $(k+1)$th step is completed in this case by defining $$u_{k+1}(x)=\frac{1}{\rho_2^2}u_k(\rho_2 x), w_{k+1}(x)=\frac{1}{\rho_2^2}w_k(\rho_2 x),$$
$$\alpha_{k+1}=\alpha', \beta_{k+1}=\beta', a_{k+1}=a', b_{k+1}=b' \text{ and } r_{k+1}=\rho_2r_k.$$

Note that in this case, \begin{equation}\label{remark1}|\alpha_{k+1}-\beta_{k+1}|=2\delta\eps_{k+1}^{1/2}.\end{equation} Consequently, either the scheme terminates in the next step, or we again fall into the third case.

This completes the description of the iteration scheme. 

If  we always end up in the second or the third case, then this scheme continues indefinitely. If at some step, the parameters fall into the first case, the scheme terminates within finite steps.

\textit{Step 2: Proof of \eqref{62} when the scheme continues indefinitely.}
 
  In this case, we have $\eps_{k+1}\le \eps_k-C\eps_k^{3/2}$ for all $k$, which implies \begin{equation}\label{63}\eps_k\le C\frac{1}{k^2}.\end{equation}Together with  $|\alpha_{k+1}-\alpha_k|\le C\eps_{k}$, we have $\alpha_k\to \alpha$ for some $\alpha\in\Sph$ with \begin{equation}\label{65}|\alpha_k-\alpha|\le C/k.\end{equation}

  For each positive $r\in(0,1/2)$, find the integer $k$ such that $$r_{k+1}\le r<r_k.$$From our construction, this implies \begin{equation}\label{64}r_k\le\frac{1}{\rho_2}r\text{ and }k\ge \log_{\rho_1}r.\end{equation}
  
From $(u,w)\in\mathcal{R}(\alpha_k,\beta_k;\eps_k) \text{ in $B_{r_k}$,}$ we have  $$\Gamma_u\cap B_r\subset \{|x\cdot\alpha_k|\le C\eps_k^{1/2}r_k\}.$$Combining this with \eqref{63}, \eqref{65} and \eqref{64}, we have   $$\Gamma_u\cap B_r\subset \{|x\cdot\alpha|\le Cr(-\log r)^{-1}\}.$$

 \textit{Step 3: Proof of \eqref{62} when the scheme terminates within finite steps.}
 
 Suppose the iteration scheme terminates at step $k$, then  the first alternative in Proposition \ref{IOF1} implies that, up to a rotation, $$\Gamma_{u_k}\cap B_{\rho_0}\subset\{|x_1|\le C\eps_{k}|x'|^{1+\alpha}\},$$ where $x'$ denotes the coordinates in the directions perpendicular to $e^1.$ Consequently, $$\Gamma_u\cap B_{\rho_0r_k}\subset\{|x_1|\le C\frac{\eps_k}{r_k^\alpha}|x'|^{1+\alpha}\}.$$
 
 For $r<\rho_0r_k$, we have $k\le C\log_{\rho_2}r$, this implies  $$\Gamma_u\cap B_{r}\subset\{|x_1|\le Cr(-\log r)^{-2}\}.$$
 
 For $r\in[\rho_0r_k,1/2)$, we can apply the same argument as in Step 2 to get the desired estimate.
 \end{proof}

\begin{remark}\label{Optimality}
In general, $C^{1,\log}$-regularity of the free boundaries is optimal at points in $\Gamma_1\cap\Gamma_2.$ 

Suppose $0\in\Gamma_1\cap\Gamma_2$, then we are always in the second or the third case as in the proof of Lemma \ref{LongIteration1}. If we are ever in the third case at one step, then \eqref{remark1} implies that we are always in the third case for later steps.  From here, \eqref{remark2} and \eqref{remark1} imply that, up to a rotation, $$c\frac{1}{k}\le|\alpha_k-e^1|\le C\frac{1}{k}.$$

As a result, for all small $r$, we can  find $x\in\Gamma_1\cap B_r$ such that $$|x_1|\ge cr(-\log r)^{-1}.$$ The free boundary $\Gamma_1$ is not better than $C^{1,\log}$ at $0$. 

In the following, we show that this is actually the \textit{generic} behavior at points in $\Reg.$
\end{remark} 

We show this generic behavior in $\R^2$. In general dimensions, the argument is similar. 

To state the generic condition, we introduce two parameters for functions defined the the sphere. For a continuous function $f:\mathbb{S}^1\to\R$, let $h$ be the solution to 
$$\begin{cases}
\Delta h=0 &\text{ in $B_1\cap\{x_1>0\}$,}\\ h=f &\text{ on $\partial B_1\cap\{x_1>0\}$,}\\ h=0 &\text{ on $B_1\cap \{x_1=0\}.$}
\end{cases}$$Define $$\gamma_1(f):=\frac{\partial}{\partial x_1}h(0), \text{ and }\gamma_2(f):=\frac{\partial^2}{\partial x_1\partial x_2}h(0).$$ Then we have the following:

\begin{proposition}\label{GenericOptimal}
Let $\varphi,\psi:\mathbb{S}^1\to\R$ be two continuous functions with $|\varphi|+|\psi|\le 1,$ $$\varphi\le\psi\le2\varphi\text{ on $\mathbb{S}^1\cap\{x_1\le 0\},$}$$  and $$\gamma_2(\varphi)\neq\gamma_2(\psi).$$

Suppose $(u_1,u_2,u_3)$ solves the $3$-membrane problem in $B_1$ with $$u_1=\frac{1}{2}\max\{x_1,0\}^2+\eps\varphi \text{ and } u_3=-\frac{1}{2}\max\{x_1,0\}^2-\eps\psi\text{ along $\partial B_1$.}$$ Then there are small positive constants $\eps_0$ and $r_0$, depending only on $|\gamma_2(\varphi)-\gamma_2(\psi)|$, such that for $\eps<\eps_0$, we have the following alternatives:

(1) If $|\gamma_1(\varphi)-\gamma_1(\psi)|\ge\frac{1}{2}r_0|\gamma_2(\varphi)-\gamma_2(\psi)|$, then $$\Gamma_1\cap\Gamma_2\cap B_{\frac{1}{4}r_0}=\emptyset.$$

(2) If $|\gamma_1(\varphi)-\gamma_1(\psi)|<\frac{1}{2}r_0|\gamma_2(\varphi)-\gamma_2(\psi)|$, then
$$\Gamma_1\cap\Gamma_2\cap B_{r_0}\neq\emptyset.$$
 Moreover,  the free boundaries $\Gamma_1$ and $\Gamma_2$ are no better than $C^{1,\log}$ at any points in $\Gamma_1\cap\Gamma_2\cap B_{r_0}.$
\end{proposition} 

\begin{remark}
Around a regular point, for $\gamma_2(\varphi)\neq\gamma_2(\psi),$  under all small perturbations in the directions of $\varphi$ and $\psi$,  either the two free boundaries decouple, or the free boundaries are precisely $C^{1,\log}$ at any remaining intersection.
\end{remark} 

\begin{proof}
\textit{Step 1: The free boundaries are no better than $C^{1,\log}$ at any point in $\Gamma_1\cap\Gamma_2\cap B_{r_0}$.}

For $r_0$ to be chosen, and $x_0\in\Gamma_1\cap\Gamma_2\cap B_{r_0},$ define $$u(x)=\frac{1}{r_0^2}u_1(x_0+r_0x) \text{ and } w(x)=-\frac{1}{r_0^2}u_3(x_0+r_0x).$$Then $$0\in\Gamma_u\cap\Gamma_w.$$  

With the same argument for Lemma \ref{ImprovementOfApproximation1}, we have 
\begin{equation}|u-\frac{1}{2}\max\{x\cdot\alpha-x_0\cdot\alpha+\eps\frac{\gamma_1(\varphi)}{r_0},0\}^2|<\tilde{\eps}\end{equation} and 
\begin{equation}|w-\frac{1}{2}\max\{x\cdot\beta-x_0\cdot\beta+\eps\frac{\gamma_1(\psi)}{r_0},0\}^2|<\tilde{\eps}\end{equation} in $B_1$ for some \begin{equation}\label{66}\tilde{\eps}\le C\eps(\frac{\eps^{1/2}}{r_0^2}+r_0)\text{ and }|\alpha-e^1-\eps\gamma_2(\varphi)|+|\beta-e^1-\eps\gamma_2(\psi)|\le C\eps^2.\end{equation}

With Remark \ref{Optimality}, it suffices to show that in the iteration in the proof for Lemma \ref{LongIteration1}, we will end up in the third case at some step.

Suppose not, then $0\in\Gamma_u\cap\Gamma_w$ implies that we always end up in the second case. At  the $k$th step, we have  $$|\alpha_k-\alpha|+|\beta_k-\beta|\le C\tilde{\eps}.$$ Together with \eqref{66}, this implies $$|\alpha_k-\beta_k|\ge c\eps(|\gamma_2(\varphi)-\gamma_2(\psi)|-Cr_0-C\frac{\eps^{1/2}}{r_0^2}).$$

By taking $r_0$ small such that $Cr_0<\frac{1}{4}|\gamma_2(\varphi)-\gamma_2(\psi)|$ and then $\eps_0$ small such that $C\frac{\eps_0^{1/2}}{r_0^2}<\frac{1}{4}|\gamma_2(\varphi)-\gamma_2(\psi)|,$ the previous estimate gives $$|\alpha_k-\beta_k|\ge c\eps|\gamma_2(\varphi)-\gamma_2(\psi)|>0.$$This is a contradiction as the left-hand side converge to $0$ as $k\to\infty.$

\textit{Step 2: If $|\gamma_1(\varphi)-\gamma_1(\psi)|\ge\frac{1}{2}r_0|\gamma_2(\varphi)-\gamma_2(\psi)|$, then $\Gamma_1\cap\Gamma_2\cap B_{\frac{1}{4}r_0}=\emptyset.$}

Define $u$ and $w$ as in \textit{Step 1} for $x_0=0$, then $$(u,w)\in\mathcal{R}(\alpha,\beta;a,b;\tilde{\eps}) \text{ in $B_1$}$$ with $a=-\eps\frac{\gamma_1(\varphi)}{r_0}$ and $b=-\eps\frac{\gamma_1(\psi)}{r_0}.$ 

With a similar argument as in Lemma \ref{RefinedLocalization}, we know that \begin{equation}\label{NoInterCondition} \text{ if }\Gamma_u\cap\Gamma_w\cap B_{\frac{1}{4}}\neq\emptyset, \text{ then }|a-b|\le\frac{1}{4}|\alpha-\beta|+C\tilde{\eps}.\end{equation}

On the other hand, with \eqref{66} and $|\gamma_1(\varphi)-\gamma_1(\psi)|\ge\frac{1}{2}r_0|\gamma_2(\varphi)-\gamma_2(\psi)|$, we have 
\begin{align*}
|a-b|&\ge\frac{1}{2}\eps|\gamma_2(\varphi)-\gamma_2(\psi)|\\&\ge\frac{1}{4}|\alpha-\beta|+\frac{1}{4}\eps|\gamma_2(\varphi)-\gamma_2(\psi)|-C\eps^2.
\end{align*}If we choose $r_0$ and $\eps_0$ small enough, then $\frac{1}{4}\eps|\gamma_2(\varphi)-\gamma_2(\psi)|-C\eps^2\gg\tilde{\eps}.$ 

Thus \eqref{NoInterCondition} implies $$\Gamma_u\cap\Gamma_w\cap B_{\frac{1}{4}}=\emptyset.$$

\textit{Step 3: If $|\gamma_1(\varphi)-\gamma_1(\psi)|<\frac{1}{2}r_0|\gamma_2(\varphi)-\gamma_2(\psi)|$, then $\Gamma_1\cap\Gamma_2\cap B_{r_0}\neq\emptyset.$}

Let $u$, $w$, $\alpha,\beta$, $a$ and $b$ be as in \textit{Step 2}. 

With $|\gamma_1(\varphi)-\gamma_1(\psi)|<\frac{1}{2}r_0|\gamma_2(\varphi)-\gamma_2(\psi)|$, the lines $\{x\cdot\alpha=a\}$ and $\{x\cdot\beta=b\}$ intersect at a point $\overline{x}\in B_{1/2}$. 

Consequently, if we define the following points $$p_1:=\partial B_1\cap \{x\cdot\alpha=a\}, \text{ } q_1:=\partial B_1\cap \{x\cdot\beta=b\},$$ $$p_2:=\partial B_1\cap \{(2\alpha-\beta)\cdot x=2a-b\},\text{ and }q_2:=\partial B_1\cap \{(2\beta-\alpha)\cdot x=2b-a\},$$ then $$|p_1-q_2|\ge c\eps|\gamma_2(\varphi)-\gamma_2(\psi)| \text{ and } |q_1-p_2|\ge c\eps|\gamma_2(\varphi)-\gamma_2(\psi)|.$$

Meanwhile,  the previous proposition implies that the free boundaries $\Gamma_u$ and $\Gamma_w$ are $C^1$ curves in $B_1$. Moreover, the intersection $\Gamma_u\cap\partial B_1$ consists of two points within $C\sqrt{\eps\tilde{\eps}}$ distance from $p_1$ and $p_2$ respectively. The intersection $\Gamma_w\cap\partial B_1$ consists of  two points within $C\sqrt{\eps\tilde{\eps}}$ distance from $q_1$ and $q_2$ respectively. 

If we choose $\eps_0$ and $r_0$ small such that $\tilde{\eps}\ll c\eps|\gamma_2(\varphi)-\gamma_2(\psi)|,$ the connectedness of the free boundaries implies that they intersect. See Figure \ref{Intersect}.

\begin{figure}[h]
\includegraphics[width=0.6\linewidth]{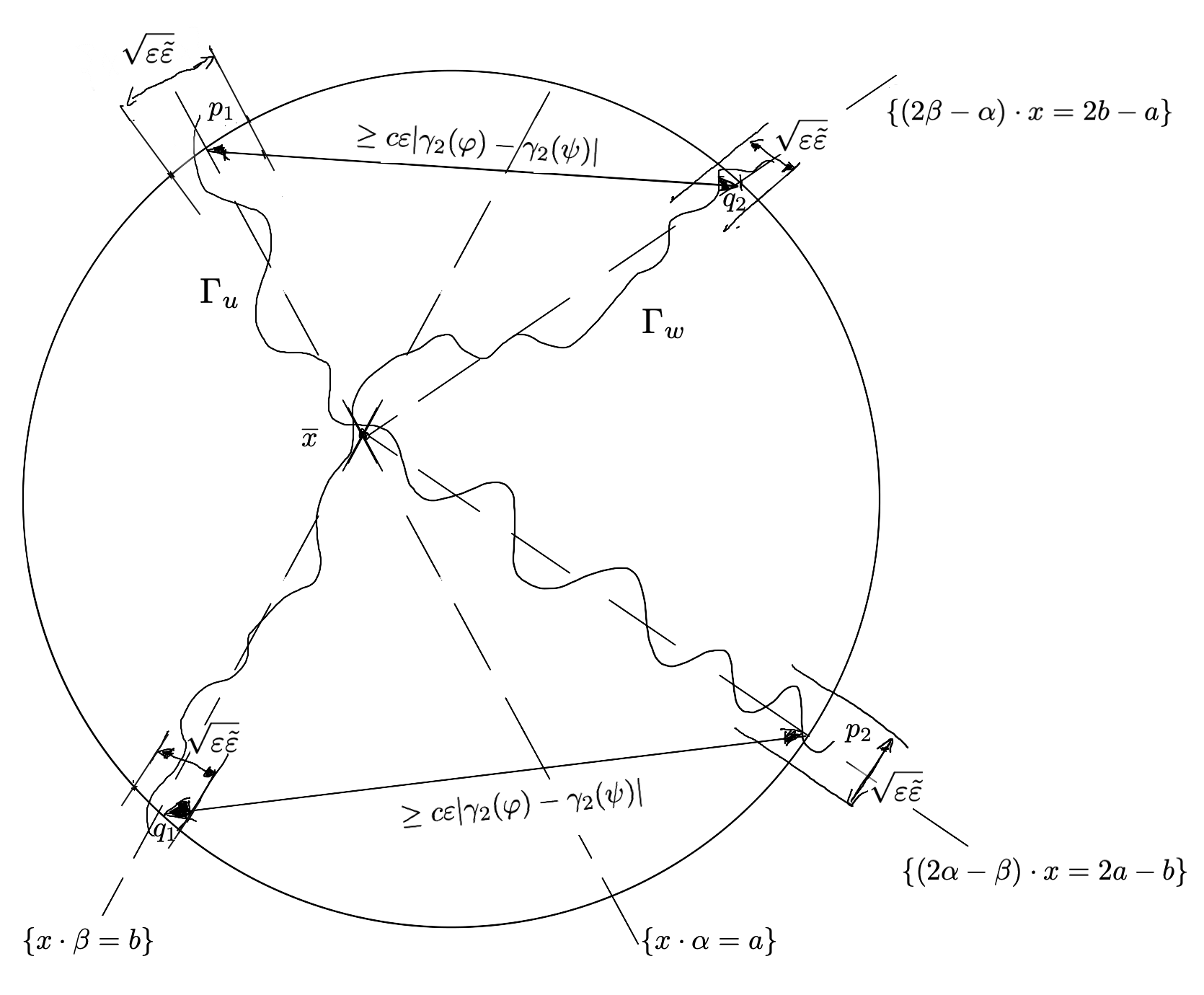}
\caption{$\Gamma_u$ and $\Gamma_w$ intersect.}
\label{Intersect}
\end{figure}
\end{proof}





\section{Free boundary regularity of $\SOne$}\label{SingularPartSection}
In this section, we prove   Theorem \ref{MainResult3}. Recall that singular points of type 1, $\SOne$,  are defined in Definition \ref{FreeBoundaryPoints}.

The proof is based on an iteration  of Proposition \ref{IOF2}. To iterate, however, the angle between the hyperplanes, namely, $|\alpha-\beta|$, has to stay below the critical level. 

This is obtained through the following lemma on Weiss energy:
\begin{lemma}\label{EnergyDrop}
Let $\Trip$ be a solution to the $3$-membrane problem in $B_1$ with $0\in\Gamma_1\cap\Gamma_2$. 

Suppose that $$|u_1-\PababTwo|+|u_3+\QababTwo|<\eps_d \text{ inside $B_1$}$$ for $\PQabab$ as in \eqref{Pabab2} and $$|\alpha-\beta|\in(\delta\eps_d^{1/2}, 2\delta\eps_d^{1/2}).$$ 
Then there is $\eta>0$ and $r\in (0,1/2)$, depending on the dimension, such that $$W((u_k),0,r)\le W((u_k),0,\frac{1}{2})-\eta.$$ Here $\eps_d$ is a dimensional constant, and $\delta$ is the dimensional constant in Proposition \ref{IOF2}.
\end{lemma} The functional $W$ is the Weiss energy defined in \eqref{WeissEnergy}.

\begin{proof}
Suppose, on the contrary, that there is no such $\eta$ and $r$. 

Then we can find a sequence of solutions $(u_k)_n$, satisfying the assumptions, but $$W((u_k)_n,0,r_n)\ge W((u_k)_n,0,\frac{1}{2})-\frac{1}{n}$$ with $r_n\to 0$ for all $n$. 

With $0\in\Gamma_1\cap\Gamma_2$, Theorem \ref{OptimalRegularity} gives compactness to the sequence. That is,  up to a subsequence, $$(u_k)_n\to (u_k) \text{ in $C^{1,\alpha}_{loc}(B_1)$}$$ for a solution $(u_k)$ in $B_1$ satisfying \begin{equation}\label{ApproxByOne}|u_1-\PababTwo|+|u_3+\QababTwo|\le\eps_d \text{ in $B_1$}\end{equation} for some $|\alpha-\beta|\in[\delta\eps_d^{1/2}, 2\delta\eps_d^{1/2}]$.

As a result, we can apply Lemma \ref{ImprovementOfAngle2} to $u=u_1$ and $w=-u_3$ to get 
\begin{equation}\label{ApproxByTwo}|u_1-P(\alpha',\beta')|<\frac{1}{2}\eps_d\rho_2^2 \text{ in $B_{\rho_2}$}\end{equation} for $P(\alpha',\beta')$ as in \eqref{Pabab2} and $$|\alpha'-\beta'|\ge|\alpha-\beta|+20\eps_d.$$

On the other hand, uniform convergence of the gradient in $B_{1/2}$ implies $$W((u_k),0,\frac{1}{2})\le W((u_k),0),$$ where $W((u_k),0)$ is the limit defined in \eqref{LimitOfWeiss}. 

Thanks to Theorem \ref{Monotonicity}, we know that $(u_k)$  is $2$-homogeneous in $B_{1/2}$.

Together with \eqref{ApproxByTwo}, this homogeneity implies that for $x\in B_{1}$,
\begin{align*}
|u_1(x)-\frac{1}{2}\min\{x\cdot\alpha'-\frac{a'}{\rho_2},0\}^2-\frac{1}{4}\max\{ x\cdot\beta'-\frac{b'}{\rho_2},0\}^2|\le\frac{1}{2}\eps_d.
\end{align*} With \eqref{ApproxByOne}, we have 
\begin{align*}|&\TopHalfTwo\\&-\frac{1}{2}\min\{x\cdot\alpha'-\frac{a'}{\rho_2},0\}^2-\frac{1}{4}\max\{ x\cdot\beta'-\frac{b'}{\rho_2},0\}^2|\le\frac{3}{2}\eps_d\end{align*} in $B_1$. 

By Lemma \ref{RefinedLocalization}, $|a|,|b|,|\frac{a'}{\rho_2}|,|\frac{b'}{\rho_2}|\le C\eps_d$.
The estimate above contradicts $|\alpha'-\beta'|\ge|\alpha-\beta|+20\eps_d$ if $\eps_d$ is small. \end{proof} 

This implies the angles remain below the critical level along $\SOne$:
\begin{lemma}\label{ProtectAngle}
Suppose $\Trip$ solves the $3$-membrane problem in $B_1$ with $$0\in\SOne$$ and $$W((u_k),0,1)\le W_1+\frac{1}{2}\eta.$$ If, for some $\eps<\eps_d$, $$|u_1-\PababTwo|+|u_3+\QababTwo|<\eps \text{ in $B_1$}$$for $\PQabab$ as in \eqref{Pabab2} with $|\alpha-\beta|<2\delta\eps^{1/2},$ then $$|\alpha-\beta|<\delta\eps^{1/2}.$$

Here $\delta$, $\eps_d$ and $\eta$ are dimensional constants. 
\end{lemma} 

Recall that $W_1$ is the Weiss energy of unstable half-space solutions as in \eqref{FirstEnergy}.

\begin{proof} Let $\eta$ be the constant from Lemma \ref{EnergyDrop}.

Suppose $|\alpha-\beta|\ge\delta\eps^{1/2}.$
We get a contradiction by iterating Lemma \ref{ImprovementOfAngle2}.

Define $u_0=u_1$, $w_0=-u_3$, $\alpha_0=\alpha$, $\beta_0=\beta$ and $\eps_0=\eps$. Then $$(u_0,w_0)\in\mathcal{S}(\alpha_0,\beta_0;\eps_0) \text{ in $B_1$,}$$where the class of solutions $\mathcal{S}$ is defined in Definition \ref{SingSol}

Suppose, for some $(m-1)$, we have found $$(u_{m-1},w_{m-1})\in\mathcal{S}(\alpha_{m-1},\beta_{m-1};\eps_{m-1}) \text{ in $B_1$}$$ with $|\alpha_{m-1}-\beta_{m-1}|\in[\delta\eps_{m-1}^{1/2},2\delta\eps_{m-1}^{1/2}].$ 

If $\eps_{m-1}\ge\eps_d$, then we terminate the iteration. 

Otherwise, we apply Lemma 5.5  to get 
 $$(u_{m-1},w_{m-1})\in\mathcal{S}(\alpha',\beta';\eps_{m-1}) \text{ in $B_{\rho_2}$}$$ with $|\alpha'-\beta'|\ge|\alpha_{m-1}-\beta_{m-1}|+20\eps_{m-1}.$ 

Define $$u_{m}=\frac{1}{\rho_2^2}u_{m-1}(\rho_2x) \text{ and }w_{m}=\frac{1}{\rho_2^2}w_{m-1}(\rho_2x), $$ $\alpha_m=\alpha'$, $\beta_m=\beta'$ and $\eps_m=(\frac{|\alpha'-\beta'|}{\delta})^2.$

Then we have $$(u_{m},w_{m})\in\mathcal{S}(\alpha_{m},\beta_{m};\eps_{m}) \text{ in $B_1$}$$
 with $$|\alpha_m-\beta_m|=\delta\eps_m^{1/2}.$$  Moreover, $$\eps_m\ge\eps_{m-1}+\frac{40}{\delta}\eps_{m-1}^{3/2}.$$

In particular, within finite steps, we  have $\eps_m>\eps_d$, and the iteration terminates. 

At the final step,  we have $$(u_{m},w_{m})\in\mathcal{S}(\alpha_{m},\beta_{m};\eps_{m}) \text{ in $B_1$}$$ with $\frac{1}{2}\eps_d<\eps_m<\eps_d$, and $|\alpha_m-\beta_m|\in[\delta\eps_m^{1/2},2\delta\eps_m^{1/2}].$

Consequently, Lemma \ref{EnergyDrop} gives $$W((u_k)_m,0,r)\le W((u_k)_m,0,\frac{1}{2})-\eta.$$
Rescale back and use the monotonicity of Weiss energy, we have 
$$W((u_k),0,r\rho_2^m)\le W((u_k),0,1)-\eta\le W_1-\frac{1}{2}\eta.$$

Meanwhile, $0\in\SOne$ implies $W((u_k),0)=W_1>W((u_k),0,r\rho_2^m)$, contradicting the monotonicity of the Weiss functional. \end{proof} 

\begin{remark}\label{SOneNotStable}
One consequence of this lemma is that $\SOne$ is generically unstable.  

Consider perturbations as in Proposition \ref{GenericOptimal}, but for unstable half-space solutions, that is, solutions satisfying   the following along $\partial B_1$: $$u_1=\frac{1}{2}\min\{x_1,0\}^2+\frac{1}{4}\max\{x_1,0\}^2 +\eps\varphi, \text{ and } u_3=-\frac{1}{4}\min\{x_1,0\}^2-\frac{1}{2}\max\{x_1,0\}^2-\eps\psi.$$ 

Define $u=u_1$ and $w=-u_3$. With similar arguments as in the proof of Proposition \ref{GenericOptimal},  $(u,w)$ are well-approximated, within error $\tilde{\eps}$,  by unstable half-space profiles   with $$|\alpha-\beta|\sim|\gamma_2(\varphi)-\gamma_2(\psi)|> \delta\tilde{\eps}^{1/2}.$$ Lemma \ref{ProtectAngle} then says that $\SOne\cap B_{r_0}$ is empty for a small $r_0.$ 

In particular, with the complete classification of homogeneous solutions in two dimensions,  $\Gamma_1\cap\Gamma_2\cap B_{r_0}$ consists of regular points after this perturbation in $\R^2.$
\end{remark} 

The previous lemma says that the angle stays strictly below the critical level. Consequently, iterations of Proposition \ref{IOF2} can be performed indefinitely. This leads to the following point-wise estimate at points in $\SOne$.

Theorem \ref{MainResult3} is a direct consequence of this point-wise localization.
\begin{lemma}\label{LongIteration2}
Suppose $\Trip$ solves the $3$-membrane problem in $B_1$ with $$0\in\SOne$$ and $$W((u_k),0,1)\le W_1+\frac{1}{2}\eta.$$ If $$|u_1-\PababTwo|+|u_3+\QababTwo|<\eps \text{ in $B_1$,}$$ for $\PQabab$ as in \eqref{Pabab2}, $|\alpha-\beta|<2\delta\eps^{1/2}$ and $\eps<\eps_d,$ then, up to a rotation, $$\Gamma_1\cap\Gamma_2\cap B_r\subset\{|x_1|\le Cr^{1+\alpha_d}\}$$ for all $r\in (0,1/2)$. 

Here $\delta$, $\eps_d$, $\alpha_d$, $\eta$ and $C$ are dimensional constants.\end{lemma} 

\begin{proof}
This proof is based on an iteration of Proposition \ref{IOF2}. 

Define $u=u_1$ and $w=-u_3$. It suffices to prove  $$\Gamma_{u}\cap\Gamma_w\cap B_r\subset\{|x_1|\le Cr^{1+\alpha_d}\}$$ for all $r\in (0,1/2)$. 

Define $u_0=u$, $w_0=w$, $\alpha_0=\alpha$, $\beta_0=\beta$, and $\eps_0=\eps$. Then applying the previous lemma, we have  $$(u_0,w_0)\in\mathcal{S}(\alpha_0,\beta_0;\eps_0) \text{ in $B_1$}$$ with $|\alpha_0-\beta_0|<\delta\eps_0^{1/2}.$

Suppose we have completed the $m$th step of this iteration, that is, we have found 
$$(u_m,w_m)\in\mathcal{S}(\alpha_m,\beta_m;\eps_m) \text{ in $B_1$}$$ with $|\alpha_m-\beta_m|<\delta\eps_m^{1/2}$. 

Then we apply Proposition \ref{IOF2} to get 
$$(u_m,w_m)\in\mathcal{S}(\alpha',\beta';\frac{1}{2}\eps_m) \text{ in $B_{\rho_1}$}$$ with $|\alpha'-\beta'|<|\alpha_m-\beta_m|+C\eps_m$. 

Define$$(u_{m+1},w_{m+1})=\frac{1}{\rho_1^2}(u_m,w_m)(\rho_1x),$$ $\alpha_{m+1}=\alpha'$, $\beta_{m+1}=\beta'$ and $\eps_{m+1}=\frac{1}{2}\eps_m$.  Then 
$$(u_{m+1},w_{m+1})\in\mathcal{S}(\alpha_{m+1},\beta_{m+1};\eps_{m+1}) \text{ in $B_1$.}$$

Also, $$|\alpha_{m+1}-\beta_{m+1}|<\delta\eps_{m}^{1/2}+C\eps_m\le \delta(2\eps_{m+1})^{1/2}+C\eps_{m+1}.$$ Consequently, we  have $|\alpha_{m+1}-\beta_{m+1}|<2\delta\eps_{m+1}^{1/2}$ if $\eps_d$ is small. 

Moreover, if we define $\tilde{u}_1=u_{m+1}$, $\tilde{u}_2=-u_{m+1}+w_{m+1}$ and $\tilde{u}_3=-w_{m+1}$, then $$W((\tilde{u}_k),0,1)=W((u_k),0,\rho_1^{m+1})\le W((u_k),0,1)\le W_1+\frac{1}{2}\eta.$$ In particular, Lemma \ref{ProtectAngle} gives $$|\alpha_{m+1}-\beta_{m+1}|<\delta\eps_{m+1}^{1/2}.$$
This completes the $(m+1)$th step of the iteration.

Now for $r\in(0,1/2)$, find the integer $m_0$ such that $\rho_1^{m_0+1}\le r<\rho_1^{m_0}$.

The estimate at step $m_0$ implies that  $$\Gamma_{u}\cap\Gamma_w\cap B_{r}\subset\{|x\cdot\alpha|\le C\frac{1}{2^{m_0/2}}\eps^{1/2}r\}$$ for some $\alpha\in\Sph.$ Using $\rho_1^{m_0+1}\le r$, we have
 $$\Gamma_{u}\cap\Gamma_w\cap B_{r}\subset\{|x\cdot\alpha|\le C\eps^{1/2}r^{1+\alpha_d}\},$$where $\alpha_d$ depends on $\rho_1.$ 
\end{proof} 





\section{Free boundary regularity of $\STwo$}
In this section, we prove Theorem \ref{MainResult1} about the stratification of singular points of type 2, $\STwo,$ as in Definition \ref{FreeBoundaryPoints}.

The proof is an application of the classical ideas of Monneau \cite{M}. It suffices to prove the following monotonicity formula at points in $\STwo$. The rest follows exactly like in \cite{M}. 

The reader is encouraged to consult Colombo-Spolaor-Velichkov \cite{CSV}, Figalli-Serra \cite{FSe} for recent developments on the singular set in the classical obstacle problem, and to consult Savin-Yu \cite{SY2, SY3} for regularity of the singular set in the fully nonlinear obstacle problem. 

\begin{lemma}
Suppose that  $\Trip$ solves the $3$-membrane problem in $B_1$ with $0\in\STwo.$

Let $(v_1,v_2,v_3)$ be a parabola solution as in Definition \ref{ParabolaSolution}.

Then the following is a non-decreasing function in $r\in(0,1):$
$$M(r)=\frac{1}{r^{d+3}}\int_{\partial B_r}\sum(u_k-v_k)^2.$$
\end{lemma} 

\begin{proof}
Define $w_k=u_k-v_k$ for $k=1,2,3.$

Let $W$ denote the Weiss energy  as in \eqref{WeissEnergy}. 

By its monotonicity as in Theorem \ref{Monotonicity}, and the definition of $\STwo$, we have 
\begin{align*}
0\le& 2W((u_k),0,r)-2W((v_k),0,r)\\=&\frac{1}{r^{d+2}}\int_{B_r}\sum(|\nabla u_k|^2-|\nabla v_k|^2)+2w_1-2w_3-\frac{2}{r^{d+3}}\int_{\partial B_r}\sum (u_k^2-v_k^2)\\=&\frac{1}{r^{d+2}}\int_{B_r}\sum(|\nabla w_k|^2+2\nabla w_k\cdot\nabla v_k)+2w_1-2w_3-\frac{2}{r^{d+3}}\int_{\partial B_r}\sum (w_k^2+2w_kv_k)\\=&\frac{1}{r^{d+2}}\int_{B_r}\sum (-w_k\Delta w_k-2w_k\Delta v_k)+2w_1-2w_3\\&-\frac{2}{r^{d+3}}\int_{\partial B_r}\sum (w_k^2+2w_kv_k)+\frac{1}{r^{d+2}}\int_{\partial B_r}\sum w_k(w_k)_\nu+\frac{2}{r^{d+2}}\int_{\partial B_r}\sum w_k(v_k)_\nu,
\end{align*}where $(\cdot)_\nu$ denotes the normal derivative of a function. 

By definition of parabola solutions, we have  $\Delta v_1=1$ and $\Delta v_3=-1$. Their homogeneity implies $r(v_k)_\nu=2v_k$ along $\partial B_r$. Thus we can continue the previous estimate to get 
\begin{align*}
0\le \frac{1}{r^{d+2}}\int_{B_r}\sum (-w_k\Delta w_k)+\frac{1}{r^{d+3}}\int_{\partial B_r}\sum w_k[r(w_k)_\nu-2w_k].
\end{align*}
Note that $$r\frac{d}{dr}M(r)=2\frac{1}{r^{d+3}}\int_{\partial B_r}\sum w_k[r(w_k)_\nu-2w_k],$$ it suffices to show that $\sum w_k\Delta w_k\ge 0.$

We actually verify this condition for general $N$, that is, when there are an arbitrary number of membranes. See the following remark.
\end{proof} 

\begin{remark}
It is interesting to note that a similar proof works when there are arbitrary number of membranes.

Suppose that $(u_1,u_2,\dots,u_N)$ solves the $N$-membrane problem with constant forcing terms $f_1>f_2>\dots>f_N,$ and that $v_k=\frac{1}{2}x\cdot A_kx$ are parabola solutions satisfying $v_k\ge v_{k+1}$ and $trace(A_k)=f_k.$ To extend the previous proof for this situation, the only non-trivial step is to show that $$\sum(u_k-v_k)\Delta( u_k-v_k)\ge0.$$  

Suppose for some $m,n$, we have, at a point $x$,  $$u_{n}(x)>u_{n+1}(x)=u_{n+2}(x)=\dots=u_{n+m}(x)>u_{n+m+1}(x).$$ Then we have $\sum_{n+1}^{n+m}\Delta(u_k-v_k)(x)=0$ and $\Delta u_k(x)=\frac{1}{m}\sum_{n+1}^{n+m}f_j$ for each $k=n+1,n+2,\dots,m$, which imply \begin{align*}&\sum_{k=n+1}^{n+m}(u_k-v_k)(x)\Delta( u_k-v_k)(x)\\=&-\sum_{k=n+1}^{n+m}v_k(x)\Delta( u_k-v_k)(x)\\=&\sum_{k=n+1}^{n+m}v_k(x)f_k-\sum_{k=n+1}^{n+m}v_k(x)(\frac{1}{m}\sum_{j=n+1}^{n+m}f_j).\end{align*} By the rearrangement inequality, this is non-negative since $v_k\ge v_{k+1}$ and $f_k\ge f_{k+1}.$\end{remark}

\appendix
\section{Free boundary regularity in the obstacle problem}
This appendix is devoted to the study of the obstacle problem, namely, 
\begin{equation}\label{OBP}\begin{cases}
\Delta u=\mathcal{X}_{\{u>0\}}\\
u\ge 0
\end{cases} \text{ in $\Omega.$}\end{equation} The goal is to show that the free boundary $\partial\{u>0\}$ is regular when the solution is well-approximated by a half-space solution. 

In essence, this is the  classical result by Caffarelli \cite{C}. However, for our purpose, we need a version with a quantified $C^{1,\alpha}$-estimate. This seems difficult to find in the literature. We include it here with a  proof. Our proof is different from the one in \cite{C}. A similar argument was used in \cite{B}.

\begin{theorem}\label{ObReg}
Suppose $u$ solves the obstacle problem \eqref{OBP} in $B_1$ with $0\in\partial\{u>0\}$. 

If  we have, for some $\eps<\eps_d$, $$|u-\frac{1}{2}\max\{x_1-a,0\}^2|<\eps \text{ in $B_1$,}$$ then $\partial\{u>0\}$ is a $C^{1,\alpha}$-hypersurface in $B_{1/2}$ with $C^{1,\alpha}$-norm  bounded by $C\eps$. 

Here $\eps_d$ and $C$ are dimensional constants.
\end{theorem} 

The proof is based on an improvement of flatness argument. To simplify our notations, we introduce the following class of solutions:
\begin{definition}
For $\alpha\in\Sph$ and $a\in\R$, we write $$u\in\mathcal{R}(\alpha;a;\eps) \text{ in $B_r$}$$
if $u$ solves the obstacle problem in $B_r$ with $0\in\partial\{u>0\}$, and $$|u-\frac{1}{2}\max\{x\cdot\alpha-a,0\}^2|<\eps r^2 \text{ in $B_r.$}$$\end{definition} 

Similar to  Lemma \ref{NonDegeneracy}, we have 
\begin{lemma}\label{NonDegenObstacle}
Suppose $u$ solves the obstacle problem in $B_r$.  

If $u\le\frac{1}{4d}r^2$ along $\partial B_r$, then $u(0)=0.$ \end{lemma}
Similar to Lemma \ref{TrappingByTranslations}, we have 
\begin{lemma}\label{TrappingByTranslationsObstacle}
Suppose $u\in\mathcal{R}(\alpha;a;\eps)$ in $B_1$. 

Define $P=\frac{1}{2}\max\{x\cdot\alpha-a,0\}^2$. 

Then there are dimensional constants $A$ and $\eps_d$ such that 
$$P(\cdot-A\eps\alpha)\le u\le P(\cdot+A\eps\alpha) \text{ in $B_{1/2}$}$$if $\eps<\eps_d.$
\end{lemma} 

Theorem \ref{ObReg} is a direct consequence of the following improvement-of-flatness result.
\begin{lemma}
Suppose $u\in\mathcal{R}(\alpha;a;\eps)$ in $B_1$ for some $\eps<\eps_d$. 

Then there are $\alpha'\in\Sph$ and $a'\in\R$, satisfying $|\alpha'-\alpha|+|a'-a|\le C\eps$, such that $$u\in\mathcal{R}(\alpha';a';\frac{1}{2}\eps) \text{ in $B_\rho$.}$$ Here $\eps_d$, $\rho$ and $C$ are dimensional constants. 
\end{lemma} 

\begin{proof}
It suffices to prove the result when $\alpha=e^1$. 

Similar to Lemma \ref{Localization1}, we have $|a|<C\eps^{1/2}.$

Define $P=\frac{1}{2}\max\{x_1-a,0\}^2$ and $\hat{u}=\frac{1}{\eps}(u-P)$. 

Then $|\hat{u}|\le 1$ in $B_1.$

With $u\ge P-\eps$, we have  $$\Delta\hu=0\text{ in $B_1\cap\{x_1>a+C\eps^{1/2}\}.$}$$ 

Meanwhile, Lemma \ref{TrappingByTranslationsObstacle} implies  \begin{equation}\label{A1}|\hu|\le C\eps^{1/2} \text{ in $B_{7/8}\cap\{x_1\le a+C\eps^{1/2}\}.$}\end{equation}

Let $h$ be the solution to $$\begin{cases}
\Delta h=0 &\text{ in $B_{7/8}\cap\{x_1>a+C\eps^{1/2}\}$,}\\ h=\hu &\text{ in $\partial B_{7/8}\cap\{x_1>a+C\eps^{1/2}\}$},\\ h=0  &\text{ in $ B_{7/8}\cap\{x_1=a+C\eps^{1/2}\}$}.
\end{cases}$$The previous estimate on $\hu$ gives $$|\hu-h|\le C\eps^{1/2} \text{ in $B_{7/8}\cap\{x_1>a+C\eps^{1/2}\}$.}$$

Using the definition of $\hu$, \eqref{A1} and boundary regularity of $h$, this leads to
$$|u-[\frac{1}{2}(x_1-a)^2+\eps(x_1-a)(\gamma_1+\sum_{k\ge 2}\gamma_kx_k)]|\le C\eps(\eps^{1/2}+r^3) \text{ in $B_r\cap\{x_1>a\}$}$$ for some  bounded constants $\gamma_k$ and for all $r<1/2$. 

If we define $\alpha'=\frac{e^1+\eps\sum_{k\ge 2}\gamma_ke^k}{|e^1+\eps\sum_{k\ge 2}\gamma_ke^k|}$ and $a'=a-\eps\gamma_1$, then $|\alpha'-\alpha|+|a'-a|\le C\eps$, and 
$$|u-\frac{1}{2}\max\{x\cdot\alpha'-a',0\}^2|\le C\eps(\eps^{1/2}+r^3) \text{ in $B_r$.}$$To conclude, we first choose $\rho$ small such that $C\rho^3\le\frac{1}{4}\rho^2$, then $\eps_d$ small such that $C\eps_d^{1/2}<\frac{1}{4}\rho^2$.
\end{proof} 

\section{An auxiliary function}
In this section, we study an auxiliary function that is useful for the two arguments for improvement of angles. To be precise, our result reads
\begin{proposition}\label{AuxiliaryFunct}
Let $H$ be the solution to the following equation
$$\begin{cases}
\Delta H=0 &\text{ in $\R^d\cap\{x_1>0\}$,}\\
H=(x_2)^2\mathcal{X}_{\{0<x_2<1\}} &\text{ on $\{x_1=0\}$,}\\
\lim_{|x|\to\infty}H=0.
\end{cases}$$
Then there are two positive constants $A_1$ and $A_2$ such that, for each $0<r<1/2$, we have $$|H-A_1x_1+A_2x_1x_2\log r|\le Cr^2 \text{ in $B_r\cap\{x_1>0\}.$}$$Here $C$ is a dimensional constant. 
\end{proposition} 

Note that $H$ depends only on the variables $x_1$ and $x_2$. Thus it suffices to consider the problem in $\R^2.$ To simplify our notations, we write $$\R^2_+=\R^2\cap\{x_1>0\}.$$

We build $H$ from dyadic blocks. The basic building block is the following:
\begin{lemma}
Let $h_0$ be the solution   to the following equation
$$\begin{cases}
\Delta h_0=0 &\text{ in $\R^2_+$,}\\
h_0=(x_2)^2\mathcal{X}_{\{1<x_2<2\}} &\text{ on $\{x_1=0\}$,}\\
\lim_{|x|\to\infty}h_0=0.
\end{cases}$$
Then $0\le h_0\le 4$, $\frac{\partial}{\partial x_1}h_0(0)>0$ and $\frac{\partial^2}{\partial x_1\partial x_2}h_0(0)>0.$\end{lemma} 

\begin{proof}The bound on $h_0$ follows  from the maximum principle.

The strong maximum principle implies that $h_0$ is strictly positive in $\R^2_+$. With $h_0(0)=0$, we have $\frac{\partial}{\partial x_1}h_0(0)>0.$ 

Let $g$ be the solution to $$\begin{cases}
\Delta g=0 &\text{ in $\R^2_+$,}\\
g=2x_2\mathcal{X}_{\{1<|x_2|<2\}} &\text{ on $\{x_1=0\}$,}\\
\lim_{|x|\to\infty}g=0.
\end{cases}$$
By symmetry, $g(x_1,0)=0$ for all $x_1>0.$

Note that $\frac{\partial}{\partial x_2}h_0$ solves  the same equation, except that it assumes boundary value $2x_2\mathcal{X}_{\{1<x_2<2\}}$ along $\{x_1=0\}.$ Thus $\frac{\partial}{\partial x_2}h_0\ge g$ along $\{x_1=0\}.$ The strong maximum principle implies   $$\frac{\partial}{\partial x_2}h_0(x_1,0)>g(x_1,0)>0 \text{ for $x_1>0$.}$$

With $\frac{\partial}{\partial x_2}h_0(0)=0,$ we have $\frac{\partial^2}{\partial x_1\partial x_2}h_0(0)>0.$\end{proof} 

Now we give the proof of Proposition \ref{AuxiliaryFunct}:
\begin{proof}
 It is elementary that $H$ can be decomposed as the following series $$H(x)=\sum_{k\ge1}\frac{1}{4^k}h_0(2^kx).$$

For $r\in(0,1/2)$, with $0\le h_0\le 4$ we have $$\sum_{k\ge\log_{\frac{1}{2}}r}\frac{1}{4^k}h_0(2^kx)\le 4\sum_{k\ge\log_{\frac{1}{2}}r}\frac{1}{4^k}\le \frac{16}{3}r^2.$$

Meanwhile, for $x\in B_{r/2}$ and $k<\log_{\frac{1}{2}}r$, we have $|2^kx|<1/2$. Boundary regularity of $h_0$ implies $$\frac{1}{4^k}|h_0(2^kx)-\frac{\partial}{\partial x_1}h_0(0)2^kx_1-\frac{\partial^2}{\partial x_1\partial x_2}h_0(0)4^kx_1x_2|\le C\frac{1}{4^k}|2^kx|^3\le C2^k|x|^3.$$Note that we have used $\frac{\partial}{\partial x_2}h_0(0)=\frac{\partial^2}{\partial x_1^2}h_0(0)=\frac{\partial^2}{\partial x_2^2}h_0(0)=0.$

Define $A_1=\frac{\partial}{\partial x_1}h_0(0)$, and $A_2=\frac{\partial^2}{\partial x_1\partial x_2}h_0(0)$. Then the previous estimate implies \begin{equation*}|\sum_{1\le k<\log_{\frac{1}{2}}r}\frac{1}{4^k}h_0(2^kx)-\sum_{1\le k<\log_{\frac{1}{2}}r}(\frac{1}{2^k}A_1x_1+A_2x_1x_2)|\le C\sum_{1\le k<\log_{\frac{1}{2}}r}2^k|x|^3\le Cr^2.\end{equation*}

Combining these, we have $$|H-A_1x_1\sum_{1\le k<\log_{\frac{1}{2}}r}\frac{1}{2^k}-A_2x_1x_2\log_{\frac{1}{2}}r|\le Cr^2 \text{ in $B_{r/2}\cap\{x_1>0\}$.}$$\end{proof}


\end{document}